\documentclass[12pt,reqno]{amsart}
\usepackage{amsmath}
\usepackage{amssymb}
\usepackage{amstext}
\usepackage{a4wide}
\usepackage{graphicx}
\usepackage{physics}
\usepackage{bm}
\usepackage[numbers,sort&compress]{natbib}
\allowdisplaybreaks \numberwithin{equation}{section}
\usepackage{color}
\usepackage{cases}

\numberwithin{equation}{section}

\newtheorem{theorem}{Theorem}[section]
\newtheorem{proposition}[theorem]{Proposition}

\newtheorem{lemma}[theorem]{Lemma}
\newtheorem*{Yudovich's Theorem}{Yudovich's Theorem}

\theoremstyle{definition}

\newtheorem{definition}[theorem]{Definition}

\theoremstyle{remark}
\newtheorem{remark}[theorem]{Remark}

\begin{document}

\title
[On Arnold-type stability theorems on a sphere]{On Arnold-type stability theorems for the Euler equation on a sphere}

 \author{Daomin Cao, Guodong Wang}
 \address{State Key Laboratory of Mathematical Sciences, Academy of Mathematics and Systems
Science, Chinese Academy of Sciences, Beijing 100190, P.R. China}
\email{dmcao@amt.ac.cn}
\address{School of Mathematical Sciences, Dalian University of Technology, Dalian 116024, P.R. China}
\email{gdw@dlut.edu.cn}


\begin{abstract}
In this paper, we establish three Arnold-type stability theorems   for steady or rotating solutions of the incompressible Euler equation on a sphere. Specifically, we prove that if the stream function of a flow solves a semilinear elliptic equation with a monotone nonlinearity, then, under appropriate conditions, the flow is stable or orbitally stable in the Lyapunov sense. In particular, our theorems apply to degree-2 Rossby-Haurwitz waves. These results are achieved via a variational approach, with the key ingredient being to show that the flows under consideration
satisfy the conditions of two Burton-type stability criteria which are established in this paper. As byproducts, we obtain some sharp rigidity results for solutions of semilinear elliptic equations on a sphere.
\end{abstract}

\maketitle

\section{Introduction}\label{sec1}

The incompressible Euler equation on a sphere is a fundamental model in the study of fluid flows on the surface of a planet, which has attracted the attention of many geophysicists, meteorologists, and mathematicians over the past few decades. Among the many related issues, stability analysis of specific flow patterns is particularly important, as it is crucial for understanding the complex long-time dynamics of fluids on the sphere.

There has been extensive research in the literature on the stability of incompressible Euler flows on the sphere, with much of it focusing on the stability of certain specific flows, particularly the stability of Rossby-Haurwitz waves  and zonal solutions. See \cite{Baines,Benard,Cap,Hos,Hos2,Lorenz,Skiba0,Skiba1,Skiba2,T}.  Recently, Constantin and Germain proved a general Arnold-type stability theorem (cf. Theorem 5 in \cite{CG}) concerning steady Euler flows on the sphere, indicating noticeable differences between Arnold's stability theorems on the sphere and those in planar Euclidean domains.
Shortly thereafter, an alternative proof for that result of Constantin and Germain  was presented in \cite{CWZ} (cf. Theorem 5.7 therein).
The purpose of this paper is to provide a more thorough investigation of Arnold's stability theorems (including both the first and the second) for the incompressible Euler equation on the sphere, using a more concise and novel approach.

Arnold's stability theorems, proposed by Arnold \cite{A1,A2} in the 1960s, including the first stability theorem and the second stability theorem,  are two general theorems for determining the Lyapunov stability of ideal fluids in two-dimensional Euclidean domains. Although Arnold's stability theorems have different formulations for different domains, they can be roughly summarized as follows (cf. \cite{A3} or \cite{MP}): If the stream function $\psi$ and the vorticity $\omega$ of a steady flow satisfy
\begin{equation}\label{ar1cc}
-c_1< \frac{\nabla\omega}{\nabla\psi}< -c_2
\end{equation}
for some $c_1,c_2 > 0$, then the flow is stable (the first stability theorem);
 If  $\psi$ and $\omega$  satisfy
\begin{equation}\label{ar2cc}
c_3< \frac{\nabla\omega}{\nabla\psi}< c_4
\end{equation}
for some \emph{sufficiently small} $c_3, c_4 >0,$ then stability still holds (the second stability theorem). Note that the steady condition ensures that $\nabla\omega$ and $\nabla\psi$ are parallel, so their ratio is meaningful. For example, for planar bounded domains, we can take $c_4=\lambda_1,$  the first eigenvalue of $-\Delta$ with zero boundary condition. If the planar domain possesses certain symmetry, such as rotational or translational symmetry, then the conditions \eqref{ar1cc} and \eqref{ar2cc} can be appropriately relaxed. See \S 3.3 in \cite{MP}.
The classical stability theorems of  Arnold in Euclidean domains are proved via the energy-Casimir (EC) functional method. Specifically, one needs to construct a suitable EC functional using conserved quantities of the Euler equation and show that the EC functional is  negative-definite  or positive-definite   under the conditions  \eqref{ar1cc}  or  \eqref{ar2cc}. However, this method has some limitations. For example,  it cannot handle the endpoint case ${\nabla\omega}/{\nabla\psi}\leq \lambda_1$ in the second stability theorem for  planar bounded domains.
To handle the endpoint case, one needs  to employ a compactness argument. This has been implemented in \cite{WGu1,WZ}, where Burton's rearrangement theory \cite{BMA,BHP} and Wolansky-Ghil's supporting functional method \cite{WG1,WG2} play an important role.

For the case of the sphere, the situation is very different. First, the sphere exhibits rich symmetries, leading to some instances where Arnold-type stable solutions are actually trivial (cf. Theorem 4 in \cite{CG}).  Second, compared to the case of planar bounded domains,  the Euler equation  on the sphere possesses additional conserved quantities, allowing the endpoint value to be elevated from the first eigenvalue $\lambda_1$ to the second eigenvalue $\lambda_2=6$ (cf.  Theorem 5 in \cite{CG}).  Third, if one considers a rotating sphere, which is physically meaningful, then some difficulties may also arise in the analysis. Due to these differences, it is challenging to establish general stability theorems on the sphere.

In this paper, we establish three Arnold-type stability theorems for steady or rotating Euler flows on the sphere (Theorems \ref{a1s}, \ref{a2s2} and \ref{a2s3} in Section \ref{sec2}), with particular emphasis on the endpoint case. Note that a specific example of the endpoint case is degree-2 Rossby-Haurwitz waves \cite{Ross,Haur}, the stability of which has long been an open problem. This problem was recently solved in \cite{CWZ}.
To achieve these results, we first prove two general stability criteria (Theorems \ref{bsc1} and \ref{bsc2} in Section \ref{sec3}), which  establish  a new variational framework for studying   stability issues on the sphere. Then we show that the  Euler flows under consideration satisfy the conditions of our stability criteria,   and thus possess stability. In this process, we actually obtain variational characterizations  of these Arnold-type flows in terms of conserved quantities of the Euler equation, which in combination with symmetry considerations lead some sharp rigidity results (Theorem \ref{rgr} in Section \ref{sec2}).

 Our stability criteria are inspired by Burton's work \cite{BAR}, where some general stability results were established for steady Euler flows in a bounded simply-connected domain. Specifically, Burton proved that if the vorticity of a steady flow is a global minimizer or an isolated local maximizer of the kinetic energy relative to some rearrangement class, then the steady flow is stable.
For the case of flows on a sphere, due to extra conservation laws, one can expect better stability criteria.  Partial work in this respect was done in \cite{CWZ}. Therein, some conserved functionals were constructed by appropriately combining  conserved quantities of the Euler equation on the sphere, and the minimization/maximization problem of these functionals  relative to some rearrangement class was analyzed, yielding  some Arnold-type stability theorems. However, the method in  \cite{CWZ} is intricate to implement, as the  combination of conserved quantities could be complicated, making compactness hard to verify and the endpoint case challenging to handle. In this paper, we instead incorporate the extra conserved quantities (some linear functionals) into the constraint set (rearrangement class), which is more concise and natural.

 It is worth mentioning that  Burton-type criteria work for the rearrangement class  of any $L^p$ function, where $1<p<\infty$, thus yielding stability under the $L^p$ norm of vorticity; whereas Arnold's method relies on the positive definiteness condition of the EC functional and needs to be conducted in the $L^2$ space of vorticity, therefore only achieving stability under the $L^2$ norm of vorticity.
In addition, Burton's method can handle the stability of nonsmooth Euler flows \cite{CWCV,CWN,Wang}, to which Arnold’s method generally cannot be applied.

 This paper is organized as follows. In Section \ref{sec2}, we present the mathematical setting of this paper and state our main results. In Section \ref{sec3}, we prove two Burton-type stability criteria with a finite number of linear constraints that will be used in subsequent sections. In Section  \ref{sc4}, based on the two stability criteria, we provide  proofs of our stability theorems.  Section \ref{sc5} is dedicated to the proof of our rigidity theorem, while Section \ref{sc6} briefly discusses the case of a rotating sphere.

\section{Euler equation on a sphere and main results}\label{sec2}

\subsection{Euler equation on a sphere}
Denote by  $\mathbb S^2$  the unit 2-sphere in $\mathbb R^3$ centered at the origin,
\[\mathbb S^2=\{\mathbf x=(x_1,x_2,x_3)\in\mathbb R^3\mid x_1^2+x_2^2+x_3^2=1\}.\]
The metric on  $\mathbb S^2$ is  the standard metric induced by the ambient space $\mathbb{R}^3$. Consider the motion of an ideal
fluid of unit density on $\mathbb S^2$. In terms of the velocity field $\mathbf v$ and the scalar pressure $P$, the governing equation is the following incompressible Euler equation:
 \begin{equation}\label{eupf}
  \begin{cases}
    \partial_t\mathbf v+\nabla_{\mathbf v}\mathbf v=2\Omega x_3\mathbf v^\perp -\nabla P, &  \mathbf x\in\mathbb S^2,\,\,t\in\mathbb R, \\
    {\rm div}\, \mathbf v=0.
  \end{cases}
\end{equation}
where $\nabla_{\mathbf v}$ is the  covariant derivative along $\mathbf v$,
$\nabla$ is the gradient operator,  ${\rm div}$ is the divergence operator,
and
$\mathbf v^\perp(t,\mathbf x)=\mathbf v(t,\mathbf x)\times \mathbf x$, the  \emph{clockwise} rotation through $\pi/2$ of $\mathbf v(t,\mathbf x)$  in $T_{\mathbf x}\mathbb S^2$.
Note that we have assumed that the sphere rotates around the polar axis $\mathbf e_3=(0,0,1)$ with a constant angular speed $\Omega$, and the rotation effect is manifested as the Coriolis force $2\Omega x_3\mathbf v^\perp$. Following the method in \cite{BKM,MB,MP}, global well-posedness of \eqref{eupf} in suitable function spaces is guaranteed. For example, a detailed proof of global well-posedness in $H^s$ for any $s>2$ was established in \cite{T}.

Due to the divergence-free condition, there exists a stream function $\psi$  such that  \[\mathbf v=\nabla^\perp\psi:=(\nabla\psi)^\perp.\] The relative vorticity $\omega$ and the absolute vorticity $\zeta$ are defined by
\[\omega={\rm div}\,\mathbf v^\perp,\quad \zeta=\omega+2\Omega x_3,\]
respectively. The absolute vorticity is a measure of the rotation of fluid particles relative to the sphere, taking into account both the relative vorticity $\omega$ resulting from the internal motion of the fluid and the rotation effect of the sphere.
Note that by definition, it holds necessarily that
\[
 \int_{\mathbb S^2}\zeta \dd{\sigma}=\int_{\mathbb S^2}\omega \dd{\sigma}=0,
\]
where $\dd{\sigma}$ is the area element of $\mathbb S^2$ under the standard metric.
The stream function $\psi$ and the absolute vorticity $\zeta$ satisfy the  Poisson equation
 \begin{equation}\label{s02}
 -\Delta\psi+2\Omega x_3=\zeta,
 \end{equation}
 where $\Delta$ is the Laplace-Beltrami operator on $\mathbb S^2$.
By adding a suitable constant, we can always assume that
 \begin{equation}\label{s03}
 \int_{\mathbb S^2}\psi \dd{\sigma}=0,
 \end{equation}
thus $\psi$ is uniquely determined by $\zeta$,
\begin{equation}\label{s04}
\psi=\mathcal G(\zeta-2\Omega x_3)=\mathcal G\zeta -\Omega x_3,
\end{equation}
where $\mathcal G$ denotes the inverse of $-\Delta$ under the zero mean condition \eqref{s03} (cf. \eqref{dog01} in Section \ref{sec3}).  Here we used the fact that $-\Delta x_3=2x_3.$
 From the above discussion, the velocity field $\mathbf v$, the stream function $\psi$, and the absolute vorticity $\zeta$ are three equivalent quantities in describing the motion of an ideal fluid on $\mathbb S^2.$

This paper is mainly approached from the perspective of the absolute vorticity $\zeta$. In terms of $\zeta,$ the  Euler equation \eqref{eupf} can be written as
\[\partial_t\zeta+  \nabla^\perp(\mathcal G\zeta -\Omega x_3)\cdot\nabla \zeta=0. \tag{$V_\Omega$} \]
By direct calculation, one can verify that
\begin{equation}\label{v0v}
 \mbox{\emph{$\zeta(t,\mathbf x)$ is a solution of $(V_0)$ if and only if  $\zeta(t,\mathsf R^{\mathbf e_3}_{\Omega t}\mathbf x)$ is a solution of $(V_\Omega)$,}}
\end{equation}
where  $\mathsf R^\mathbf p_\theta$ denotes a rotation through an angle $\theta$ around some \emph{unit} vector $\mathbf p$,  which can be explicitly expressed via Rodrigues' rotation formula:
\[\mathsf R^{\mathbf p}_\theta\mathbf x=(\cos\theta)\mathbf x+\sin\theta(\mathbf p\times \mathbf x)+(1-\cos\theta)(\mathbf p\cdot\mathbf x)\mathbf p.\]
Due to \eqref{v0v}, in many cases, we only need to consider the non-rotating equation $(V_0)$.

The Euler equation on $\mathbb S^2$ possesses rich symmetries:
\begin{itemize}
  \item[(S1)]  $(V_0)$ is invariant under $\mathbb S\mathbb O(3),$ the three-dimensional rotation group,  i.e., if   $\zeta(t,\mathbf x)$ is a solution of  $(V_0)$, then $\zeta(t,\mathsf R\mathbf x)$ is also a solution of  $(V_0)$ for any $\mathsf R\in\mathbb S\mathbb O(3).$

  \item [(S2)] For general $\Omega\in\mathbb R$, $(V_\Omega)$ is invariant under $\mathbb H$, the subgroup of $\mathbb S\mathbb O(3)$ with $\mathbf e_3$ fixed,
      \begin{equation}\label{dfh1}
       \mathbb H:=\{\mathsf R\in\mathbb S\mathbb O(3)\mid \mathsf R\mathbf e_3=\mathbf e_3\}.
      \end{equation}
      More specifically,
       if $\zeta(t,\mathbf x)$ is a solution of $(V_\Omega)$, then  $\zeta(t,\mathsf H\mathbf x)$  is also a solution of $(V_\Omega)$ for any $\mathsf H\in\mathbb H.$
\end{itemize}

There are three important conservation laws for the Euler equation on $\mathbb S^2$, which will play an essential role in the sequel.
\begin{itemize}
  \item [(C1)]  The kinetic energy of the fluid
  \[\mathcal E=\frac{1}{2}\int_{\mathbb S^2}|\mathbf v|^2\dd{\sigma}\]
  is conserved. Note that by \eqref{s02} and \eqref{s04}, $\mathcal E$ can be expressed in terms of $\zeta$:
  \begin{equation}\label{s05}
   \mathcal E=\frac{1}{2}\int_{\mathbb S^2} \zeta\mathcal G\zeta \dd{\sigma}-\Omega\int_{\mathbb S^2}x_3\zeta \dd{\sigma}+\frac{4}{3}\pi\Omega^2,
  \end{equation}
  \item [(C2)]  The distribution function of $\zeta$ is invariant, i.e., $\zeta(t,\cdot)\in\mathcal R_{\zeta_0}$ for any $t\in\mathbb R,$
      where $\mathcal R_v$ denotes the  rearrangement class of a given measurable function $v$,
 \begin{equation}\label{dor1}
 \mathcal R_{v}:=\left\{ u:\mathbb S^2\mapsto\mathbb R \mid |\{\mathbf  x\in\mathbb S^2\mid u(\mathbf   x)>s\}|=|\{\mathbf   x\in\mathbb S^2\mid v(\mathbf   x)>s\}|\,\,\forall\,s\in\mathbb R\right\},
 \end{equation}
 where $|\cdot|$ is the two-dimensional Hausdorff measure.
   \item [(C3)] If $\Omega=0$, then the position-weighted integral of $\zeta$
\begin{equation}\label{dom1}
\mathbf m(\zeta):=\int_{\mathbb S^2}\mathbf x\zeta(t,\mathbf x) \dd{\sigma}
\end{equation}
  is conserved. In combination with the property \eqref{v0v}, we see that for general $\Omega\in\mathbb R,$
   \[\int_{\mathbb S^2}\mathbf x\zeta(t,\mathsf R^{\mathbf e_3}_{-\Omega t}\mathbf x) \dd{\sigma}=\mathsf R^{\mathbf e_3}_{\Omega t}\mathbf m(\zeta)\]
   is conserved. In particular,
   \[\int_{\mathbb S^2} x_3\zeta(t,\mathbf x)\dd{\sigma},\quad \left(\int_{\mathbb S^2} x_1\zeta(t,\mathbf x) \dd{\sigma}\right)^2+ \left(\int_{\mathbb S^2} x_2\zeta(t,\mathbf x) \dd{\sigma}\right)^2\]
   are both conserved for general $\Omega\in\mathbb R$.
\end{itemize}
As a consequence of (C1) and (C2),
\begin{equation}\label{doe1}
E(\zeta):=\frac{1}{2}\int_{\mathbb S^2}\zeta\mathcal G\zeta \dd{\sigma}
\end{equation}
is a conserved functional. Note that $\mathcal E=E$ if $\Omega=0$.

\subsection{Steady and rotating solution}
It is easy to check that if the stream function $\psi$ satisfies the following semilinear elliptic equation
\begin{equation}\label{s101}
-\Delta \psi+2\Omega x_3=g(\psi),
\end{equation}
then the corresponding flow is steady for $(V_\Omega)$. In terms of the absolute vorticity $\zeta,$ \eqref{s101} becomes
\begin{equation}\label{s141}
\zeta =g(\mathcal G\zeta -\Omega x_3).
\end{equation}

In this paper, we consider the following equation, which is more general than  \eqref{s141}:
\begin{equation}\label{seee}
 \zeta= g(\mathcal G \zeta+\beta\mathbf p\cdot\mathbf x),
\end{equation}
 where $\beta\in\mathbb R$ and $\mathbf p$ is a unit vector.
By direct calculation, one can check that $ \zeta(\mathsf R^{\mathbf p}_{\beta  t}\mathbf x)$ is a  solution of $(V_0)$, which is a steady solution if $\beta=0,$ and is a rotating solution around $\mathbf p$ if $\beta\neq 0$ and $\zeta$ is not
rotationally invariant about $\mathbf p.$
  Furthermore, using the property \eqref{v0v}, we see that
 for general $\Omega\in\mathbb R$, $ \zeta(\mathsf R^{\mathbf p}_{\beta t}\mathsf R^{\mathbf e_3}_{\Omega t}\mathbf x)$ is a  solution of $(V_\Omega)$, which also represents a steady or rotating solution (which may not necessarily have a fixed axis).

Typical examples of steady and rotating solutions arising from  equation \eqref{seee} include the famous Rossby-Haurwitz waves.
Recall that the $j$-th eigenvalue of $-\Delta$ on $\mathbb S^2$ under the zero mean condition \eqref{s03} is $\lambda_j=j(j+1),$  $j\in\mathbb Z_+,$
and  the  eigenspace $\mathbb E_j$ related to $\lambda_j$ has dimension $2j+1$. For example,
\[\lambda_1=2,\quad \mathbb E_1={\rm span}\{x_1,x_2,x_3\},\]
\[\lambda_2=6,\quad \mathbb E_2={\rm span}\{x_1x_2,x_1x_3,x_2x_3,x_1^2-x_2^2,3x_3^2-1\}.\]
In \eqref{seee}, if ${\mathbf p}=\mathbf e_3$ and $ g(s)=j(j+1)s$,
then $\zeta=\alpha x_3+Y$ is a solution of \eqref{seee}, where
 \begin{equation}\label{rhw007}
\quad  Y\in\mathbb E_j,\quad  \alpha:=\begin{cases}
 \frac{2\lambda_j}{2-\lambda_j}\beta&\mbox{if }j\geq 2,\\
 \mbox{arbitrary real number}& \mbox{if }j=1.
 \end{cases}
 \end{equation}
 Therefore
 \begin{equation}\label{rhw0}
 \zeta(\mathsf R^{\mathbf e_3}_{\beta  t}\mathbf x)= \alpha x_3+Y(\mathsf R^{\mathbf e_3}_{\beta  t}\mathbf x),\quad  Y\in\mathbb E_j
 \end{equation}
solves $(V_0),$ and thus
\begin{equation}\label{rhw0}
\zeta(\mathsf R^{\mathbf e_3}_{\beta t}\mathsf R^{\mathbf e_3}_{\Omega t}\mathbf x)= \alpha x_3+Y(\mathbf R^{\mathbf e_3}_{(\beta+\Omega)  t}\mathbf x),\quad  Y\in\mathbb E_j
\end{equation}
solves $(V_\Omega).$ Such solutions are called  degree-$j$ Rossby-Haurwitz waves. It is easy to see that a degree-$j$ Rossby-Haurwitz wave of the form \eqref{rhw0} is steady if and only if $\beta=-\Omega$ or $Y$ is \emph{zonal} (i.e., $Y$ depends only on $x_3$).

\subsection{Main results}

 Now we are ready to state our main results. To make the statements concise and the proofs clear, we will focus only  on the nonrotating case $\Omega=0$. The  rotating case  will be discussed in Section \ref{sc6}.

Define
\begin{equation}\label{lp0}
\mathring L^p(\mathbb S^2):=\left\{v\in L^p(\mathbb S^2)\mid \int_{\mathbb S^2}v \dd{\sigma}=0\right\}.
\end{equation}
Our first result is the following extension of Arnold's first stability theorem.

\begin{theorem}\label{a1s}
Let $1<p<\infty.$ Suppose that $\zeta\in \mathring L^p(\mathbb S^2)$ satisfies
 \begin{equation}\label{nav1}
\zeta= g(\mathcal G\zeta+\beta{\mathbf p}\cdot\mathbf x)\quad{\rm a.e.\, \,on } \,\,\mathbb S^2,
 \end{equation}
for some decreasing\footnote{Namely, $g(s_1)\geq g(s_2)$ as long as $s_1\leq s_2.$} function $g:\mathbb R\to\mathbb R\cup\{\pm \infty\}$, some unit vector  $\mathbf p\in\mathbb R^3$ and some $\beta\in\mathbb R.$
Then $\zeta$  is    $L^p$-stable   under the dynamics of $(V_0)$:  for any   $\varepsilon>0,$ there exists some $\delta>0$, such that for any smooth solution $\zeta(t,\mathbf x)$ of $(V_0)$,  it holds that
     \begin{equation}\label{scmii1}
     \|\zeta(0,\cdot)-\zeta\|_{L^p(\mathbb S^2)}<\delta\quad\Longrightarrow\quad   \|\zeta(t,\cdot)-\zeta\|_{L^p(\mathbb S^2)}<\varepsilon\quad\forall\,t\in\mathbb R.
     \end{equation}

\end{theorem}

As a straightforward consequence of Theorem \ref{a1s}, any $Y\in \mathbb E_1$ is   $L^p$-stable   under the dynamics of $(V_0)$.
In fact, write
$Y=\alpha\mathbf q\cdot\mathbf x$, where $\alpha\in\mathbb R$ and  $\mathbf q$ is a unit vector. Then $Y$ solves \eqref{nav1} with
\[g(s)=-2s,\quad \mathbf p=\mathbf q,\quad \beta=- \alpha.\]

Our second result deals with the endpoint case of Arnold's second stability theorem. Recall the definition \eqref{dor1} of the  rearrangement class of a given  measurable function. Define
\[ \zeta+\mathbb E_2=\{\zeta+Y\mid Y\in \mathbb E_2\}.\]

\begin{theorem}\label{a2s2}
Let $1<p<\infty.$  Suppose that $\zeta\in \mathring L^p(\mathbb S^2)$ satisfies
 \begin{equation}\label{nbv2}
\zeta= g(\mathcal G\zeta+\beta{\mathbf p}\cdot\mathbf x)\quad{\rm a.e.\, \,on } \,\,\mathbb S^2
 \end{equation}
 for some $g\in C^1(\mathbb R)$, some unit vector  $\mathbf p\in\mathbb R^3$ and some $\beta\in\mathbb R$. Suppose that $g$ satisfies
 \[ 0\leq  g'(s)\leq 6\quad \forall\,s\in\mathbb R.\]
Then $\zeta$ is orbitally stable  under the dynamics of $(V_0)$ in the following sense: for any $\varepsilon>0,$ there exists some $\delta>0$, such that for any smooth solution $\zeta(t,\mathbf x)$ of $(V_0)$,  if
  \[\|\zeta(0,\cdot)-\zeta\|_{L^p(\mathbb S^2)}<\delta,\]
  then  for any $t\in\mathbb R$, there exists some  $\eta_t\in (\zeta+\mathbb E_2)\cap \mathcal R_{\zeta}$ such that
  \[\| \zeta(t,\cdot)-\eta_t\|_{L^p(\mathbb S^2)}<\varepsilon.\]
     Moreover, if   additionally
     \begin{equation}\label{xy6}
     g'(s)<6\quad \forall\,s\in\mathbb R,
     \end{equation} then
     $(\zeta+\mathbb E_2)\cap \mathcal R_{\zeta}=\{\zeta\},$
      and thus $\zeta$   is  $L^p$-stable  under the dynamics of $(V_0)$ in this case.
\end{theorem}
\begin{remark}
Note that by $L^p$ estimate and Sobolev embedding, the function $\zeta$ in Theorem \ref{a2s2}, as well as in  Theorem \ref{a2s3} below, is actually  $C^1.$
\end{remark}

\begin{remark}\label{rkk8}
The conclusion of Theorem \ref{a2s2} can also be stated as follows: the set $(\zeta+\mathbb E_2)\cap \mathcal R_{\zeta}$ is  $L^p$-stable under the dynamics of $(V_0)$, i.e.,  for any $\varepsilon>0,$ there exists some $\delta>0$, such that for any smooth solution $\zeta(t,\mathbf x)$ of $(V_0)$, it holds that
    \[
     \inf_{v\in(\zeta+\mathbb E_2)\cap \mathcal R_{\zeta} }\|\zeta(0,\cdot)-v\|_{L^p(\mathbb S^2)}<\delta\quad\Longrightarrow\quad  \inf_{v\in(\zeta+\mathbb E_2)\cap \mathcal R_{\zeta} }\|\zeta(t,\cdot)-v\|_{L^p(\mathbb S^2)}<\varepsilon\quad \forall\,t\in\mathbb R.
   \]
   In combination with a continuity argument, we see that any isolated set  of $(\zeta+\mathbb E_2)\cap \mathcal R_{\zeta}$ is also $L^p$-stable under the dynamics of $(V_0)$. Here a subset $K$ of $(\zeta+\mathbb E_2)\cap \mathcal R_{\zeta}$ is called isolated if there is a positive distance (under the $L^p$ norm)  between $K$ and its  complement in $(\zeta+\mathbb E_2)\cap \mathcal R_{\zeta}$.
\end{remark}

It is an interesting problem to analyze the structure of the set $(\zeta+\mathbb E_2)\cap \mathcal R_{\zeta}$.   A typical example is the degree-2 Rossby-Haurwitz wave
\[\zeta=\alpha x_3+Y,\quad  \alpha\in\mathbb R,\quad 0\not\equiv  Y\in\mathbb E_2,\]
which corresponds to \eqref{nbv2} with
 \[g(s)=6s,\quad \mathbf p=\mathbf e_3,\quad \beta=-\frac{1}{3}\alpha.\]
 It has been proved in \cite{CWZ} that,
 \begin{itemize}
 \item [(1)] if $\alpha=0,$ then the set of all three-dimensional rigid rotations of $\zeta$ is an isolated subset of $(\zeta+\mathbb E_2)\cap \mathcal R_{\zeta}$;
 \item [(2)] if  $\alpha\neq 0,$ then the set of all  three-dimensional rigid rotations of $\zeta$  around the polar axis  is an isolated subset of $(\zeta+\mathbb E_2)\cap \mathcal R_{\zeta}$.
 \end{itemize}
 Using Remark \ref{rkk8}, we deduce that,  under the dynamics of $(V_0)$,   any $Y\in\mathbb E_2$   is $L^p$-stable up to  three-dimensional rigid rotations, and any $\alpha x_3+ Y\in\mathbb E_2$ with $\alpha\neq 0$  is $L^p$-stable up to  three-dimensional rigid rotations  around the polar axis.  It is worth mentioning that such stability results are actually sharp (cf. Section 8 in \cite{CWZ}).
However, for the general case, it seems difficult to study the structure of  $(\zeta+\mathbb E_2)\cap \mathcal R_{\zeta}$ as well as the sharpness of  orbital stability. We leave them as open problems.

In  the 1990s,  Wolansky and Ghil  \cite{WG1,WG2} proved some extensions of Arnold's second stability theorem for planar  bounded domains   using the  supporting functional method, showing that  ${\nabla\omega}/{\nabla\psi}$  being small pointwise in condition \eqref{ar2cc} can be replaced by  ${\nabla\omega}/{\nabla\psi}$ being small ``in an average sense". Our next theorem provides a similar result for the sphere.

\begin{theorem}\label{a2s3}
Let $1<p<\infty.$   Suppose that $\zeta\in \mathring L^p(\mathbb S^2)$ satisfies
 \begin{equation}\label{nbv3}
 \zeta= g(\mathcal G\zeta+\beta{\mathbf p}\cdot\mathbf x)\quad{\rm a.e.\, \,on } \,\,\mathbb S^2
 \end{equation}
 for some $ g\in C^1(\mathbb R)$, some unit vector  $\mathbf p\in\mathbb R^3$ and some $\beta\in\mathbb R$. Suppose that $g$ is increasing and $-\Delta-g'(\mathcal G \zeta+\beta{\mathbf p}\cdot\mathbf x)>0$ on $\mathbb E_1^\perp,$ i.e.,
 \begin{equation}\label{pe1}
 \int_{\mathbb S^2}|\nabla \varphi|^2-g'(\mathcal G \zeta+\beta{\mathbf p}\cdot\mathbf x)\varphi^2 \dd{\sigma} \geq c\int_{\mathbb S^2} \varphi^2 \dd{\sigma} \quad\forall\,\varphi\in \mathring C^\infty(\mathbb S^2),\,\, \mathbf m(\varphi)=\mathbf 0
 \end{equation}
 for some $c>0$, where $\mathring C^\infty(\mathbb S^2)$ is defined as in \eqref{lp0}, i.e.,
 \[
\mathring C^\infty(\mathbb S^2):=\left\{v\in   C^\infty(\mathbb S^2)\mid \int_{\mathbb S^2}v \dd{\sigma}=0\right\},
\]
and $\mathbf m$ is defined in \eqref{dom1}.
 Then  $\zeta$  is    $L^p$-stable  under the dynamics of $(V_0)$ as in \eqref{scmii1}.
\end{theorem}
\begin{remark}
In Theorem \ref{a2s3}, if $g'(s)<6$ for any $s\in\mathbb R,$ then \eqref{pe1} holds. This is a consequence of the following  Poincar\'e inequality:
\[
\int_{\mathbb S^2}\varphi^2 \dd{\sigma}\leq \frac{1}{6}\int_{\mathbb S^2} |\nabla\varphi|^2\dd{\sigma} \quad\forall\,\varphi\in \mathring C^\infty(\mathbb S^2),\,\, \mathbf m(\varphi)=\mathbf 0.
\]

\end{remark}

Our last theorem is concerned with  the rigidity properties of the solutions in Theorems \ref{a1s}, \ref{a2s2} and \ref{a2s3}.

\begin{theorem}\label{rgr}
Let $1<p<\infty.$ Suppose that $\zeta\in \mathring L^p(\mathbb S^2)$ satisfies
 \begin{equation}\label{nav13}
\zeta= g(\mathcal G\zeta+\beta{\mathbf p}\cdot\mathbf x)\quad{\rm a.e.\, \,on } \,\,\mathbb S^2,
 \end{equation}
for some  function $g$, some unit vector  $\mathbf p\in\mathbb R^3$ and some $\beta\in\mathbb R.$
\begin{itemize}
  \item [(i)] Suppose that $g$ satisfies the conditions in Theorem \ref{a1s}, i.e., $g:\mathbb R\to\mathbb R\cup\{\pm \infty\}$ is decreasing. Then $\zeta$ is rotationally invariant around some unit vector $\mathbf q\in\mathbb R^3$, i.e.,
      \[\zeta=\zeta\circ\mathsf R^{\mathbf q}_{\theta}\quad \forall\,\theta\in\mathbb R.\]
      Moreover, we can take $\mathbf q=\mathbf p$ if $\beta\neq 0$.
   \item[(ii)] Suppose that $g$ satisfies the conditions in Theorem \ref{a2s2}, i.e.,  $g\in C^1(\mathbb R)$ and
$  0\leq  g'\leq 6. $
Then $\zeta$ can be decomposed into
\begin{equation}\label{decp}
\zeta=\zeta_e+\zeta_z,
\end{equation}
where  $\zeta_e\in\mathbb E_2,$ and $\zeta_z$ is zonal up to some rigid rotation, i.e.,
there exists some unit vector $\mathbf q\in\mathbb R^3$ such that
\[\zeta_z\circ \mathsf R^{\mathbf q}_{\theta}=\zeta_z\quad \forall\,\theta\in\mathbb R.\]
 Moreover, we can take $\mathbf q=\mathbf p$ if $\beta\neq 0$.
   \item [(iii)] Suppose that $g$ satisfies the conditions in Theorem \ref{a2s3}, i.e.,  $ g\in C^1(\mathbb R)$, $g'\geq0,$  and
 $-\Delta-g'(\mathcal G \zeta+\beta{\mathbf p}\cdot\mathbf x)>0$ on $\mathbb E_1^\perp$. Then $\zeta$ is rotationally invariant around some unit vector $\mathbf q\in\mathbb R^3$.   Moreover, we can take $\mathbf q=\mathbf p$ if $\beta\neq 0$.
\end{itemize}
\end{theorem}

\begin{remark}
In terms of the stream function, Theorem \ref{rgr} provides some rigidity results
for solutions of semilinear elliptic equations on a sphere. More specifically, suppose that $\psi$ satisfies
\[\psi\in \mathring W^{2,p}(\mathbb S^2),\quad -\Delta\psi=g(\psi+\beta\mathbf p\cdot\mathbf x)\quad \mbox{on }\mathbb S^2\]
for some $p\in(1,\infty),$ some  function $g:\mathbb R\mapsto \mathbb R\cup\{\pm\infty\}$, some unit vector  $\mathbf p\in\mathbb R^3$ and some $\beta\in\mathbb R,$ where
 \[\mathring W^{2,p}(\mathbb S^2):=\left\{v\in W^{2,p}(\mathbb S^2)\mid \int_{\mathbb S^2}v \dd{\sigma} =0\right\}.\]
Then the following assertions hold:
\begin{itemize}
\item[(i)] If $g$ is decreasing, then $\psi$ is rotationally invariant around some unit vector $\mathbf q\in\mathbb R^3$,  and, moreover, we can take $\mathbf q=\mathbf p$ if $\beta\neq 0$.
\item[(ii)] If  $g\in C^1(\mathbb R)$ and
$  0\leq  g'\leq 6,$
then   
$\psi=\psi_e+\psi_z$ for some $\psi_e\in\mathbb E_2$ and some $\psi_z$ that is zonal about some unit vector $\mathbf q$. Moreover, we can take $\mathbf q=\mathbf p$ if $\beta\neq 0$.
\item[(iii)]  If $ g\in C^1(\mathbb R)$, $g'\geq0,$  and
 $-\Delta-g'(\psi+\beta{\mathbf p}\cdot\mathbf x)>0$ on $\mathbb E_1^\perp$, then $\psi$ is rotationally invariant around some unit vector $\mathbf q\in\mathbb R^3$, and, moreover, we can take $\mathbf q=\mathbf p$ if $\beta\neq 0$.
\end{itemize}
\end{remark}

\begin{remark}
Theorem \ref{rgr} provides some extensions of \cite[Proposition 1.1]{CDG}, which asserts that $\zeta$ must be identically zero if $ g\in C^1(\mathbb R)$ and $g'<2.$
\end{remark}

It is worth mentioning  that the rigidity results in Theorem \ref{rgr} are actually sharp, as can be clearly seen from the examples of degree-1 and degree-2 Rossby-Haurwitz waves.

\section{Burton-type stability criteria with linear constraints}\label{sec3}

\subsection{Burton-type stability criteria}

 In this section, we establish  a new variational framework for studying the stability of steady or rotating solutions of the Euler equation on a sphere by proving two Burton-type stability criteria.

We begin with some geometric and functional properties about rearrangements that will be used in the sequel.
Throughout this section, let $1<p<\infty$ be fixed. For $v\in L^p(\mathbb S^2)$, let   $\mathcal R_v$ be the set of rearrangements of $v$ on $\mathbb S^2$ (cf. \eqref{dor1}).
Denote by $\bar{\mathcal R}_v$ the weak closure of $\mathcal R_v$ in $L^p(\mathbb S^2).$   Define
 \[\mathcal C_{v}:=\left\{u\in L^p(\mathbb S^2)\mid\mathbf m_u=\mathbf m_v\right\}.\]
 where $\mathbf m_v$ is the position-weighted integral of $v$ defined by \eqref{dom1}.
Some important properties about $\mathcal R_v$  are listed as follows.
\begin{itemize}
\item The weak and strong topologies of $L^p(\mathbb S^2)$ coincide on $\mathcal R_v.$ This follows from the fact  that strong convergence is equivalent to weak convergence for any sequence on $\mathcal R_v$, and the fact that $\mathcal R_v$ with the weak topology is metrizable.
    \item $\bar{\mathcal R}_v$ is weakly sequentially compact in $L^p(\mathbb S^2).$ As a corollary, $\bar{\mathcal R}_v\cap\mathcal C_v$ is also weakly sequentially compact in $L^p(\mathbb S^2)$ since $\mathcal C_v$ is weakly closed in $L^p(\mathbb S^2).$
  \item $\bar{\mathcal R}_v$ is convex, and the set of extremum points of $\bar{\mathcal R}_v$ is exactly $\mathcal R_v$. See \cite[Lemmas 2.2 and 2.3]{BHP}.
    \item $\bar{\mathcal R}_v\cap\mathcal C_v$ is convex, and  the set of extremum points of $\bar{\mathcal R}_v\cap\mathcal C_v$ is exactly $\mathcal R_v\cap\mathcal C_v$. See \cite[Lemma 4.4]{BM} or \cite[Lemma 5]{BRy}. In view this fact, we 
    further deduce that the weak closure of $\mathcal R_v\cap\mathcal C_v$ is exactly 
$\bar{\mathcal R}_v\cap\mathcal C_v.$
          \item
          Let  $\mathcal R_1$ and $\mathcal R_2$ be the sets of  rearrangements  of two $L^{p}$ functions  on $\mathbb S^{2}$, respectively. Then for any $u_1\in\mathcal R_1$, there exists $u_2\in \mathcal R_2$ such that
\begin{equation}\label{kd01}
\|u_1-u_2\|_{L^p(\mathbb S^2)}=\inf_{v_1\in\mathcal R_1, v_2\in\mathcal R_2}\|v_1-v_2\|_{L^p(\mathbb S^2)}.
\end{equation}
See \cite[Lemma 2.3]{Bata}.
\end{itemize}

We also introduce some basic facts about the Green operator $\mathcal G.$
Recall that $\mathcal G$ is  the inverse of $-\Delta$ on $\mathbb S^2$ with  zero mean condition,
i.e.,
\begin{equation}\label{dog01}
-\Delta(\mathcal Gv)=v\,\,{\rm on }\,\,\mathbb S^2,\quad \int_{\mathbb S^2}\mathcal Gv \dd{\sigma}=0.
\end{equation}
The following assertions can be found in Section 3 of \cite{CWZ}:
\begin{itemize}
\item $\mathcal G$ is a bounded linear operator from $\mathring L^p(\mathbb S^2)$ onto $\mathring W^{2,p}(\mathbb S^2)$.
   \item $\mathcal G$ is  a compact linear operator from $\mathring L^p(\mathbb S^2)$ into $\mathring L^r(\mathbb S^2)$ for any $1\leq r\leq \infty.$
    \item $\mathcal G$ is symmetric and positive-definite in $\mathring L^p(\mathbb S^2)$, i.e.,
        \begin{equation}\label{smty01}
        \int_{\mathbb S^2}u\mathcal Gv\dd{\sigma}=\int_{\mathbb S^2}v\mathcal Gu\dd{\sigma}\quad\forall\,u,v\in \mathring L^p(\mathbb S^2),
        \end{equation}
          \begin{equation}\label{podf01}
          \int_{\mathbb S^2}u\mathcal Gu\dd{\sigma}\geq 0  \quad\forall\,u\in \mathring L^p(\mathbb S^2),
          \end{equation}
        and the equality in \eqref{podf01} holds if and only if $u\equiv 0.$
\end{itemize}

\begin{definition}\label{dam01}
\emph{
  Let $1<p<\infty.$ An admissible map is a map  $\zeta:\mathbb R\mapsto \mathring L^p(\mathbb S^2)$ such that
  \begin{itemize}
    \item  [(i)] $E(\zeta(t))=E(\zeta(0))$ for any $t\in\mathbb R,$ where $E$ is defined by \eqref{doe1},
    \item  [(ii)]$ \zeta(t)\in \mathcal R_{\zeta(0)}$ for any $t\in\mathbb R,$   where $\mathcal R$ is defined by \eqref{dor1},
    \item  [(iii)]$\mathbf m( \zeta(t)) =\mathbf m(\zeta(0))$ for any $t\in\mathbb R,$   where $\mathbf m$ is defined by \eqref{dom1}.
  \end{itemize}
  If $\zeta$ is additionally continuous,  it is called a continuous admissible map.}
\end{definition}
Note that any smooth solution $\zeta(t,\cdot)$ of $(V_0)$ is a continuous admissible map.

The main purpose of this section is to prove the following two theorems.
\begin{theorem}[Stability criterion: local minimizer]\label{bsc1}
 Let $1<p<\infty$ and $\zeta\in\mathring L^p(\mathbb S^2)$. Suppose that  $\zeta$  is  a local minimizer of $E$ relative to   ${\mathcal R}_\zeta\cap \mathcal C_\zeta$, i.e.,
  there exists some open neighborhood $U$ of $\zeta$ in $L^p(\mathbb S^2)$ such that
       \begin{equation}\label{etg01}
E(\zeta)\leq   E(v) \quad\forall\,v\in \mathcal R_\zeta\cap\mathcal C_\zeta\cap  U.
       \end{equation}
Then $\zeta$  is $L^p$-stable with respect to  admissible perturbations (cf. Definition \ref{dam01}), i.e., for any $\varepsilon>0,$ there exists some $\delta>0$, such that for any admissible map $\zeta(t)$, it holds that
     \begin{equation}\label{scm4}
     \|\zeta(0)-\zeta\|_{L^p(\mathbb S^2)}<\delta\quad\Longrightarrow\quad   \|\zeta(t)-\zeta\|_{L^p(\mathbb S^2)}<\varepsilon\quad\forall\,t\in\mathbb R.
     \end{equation}
  \end{theorem}

\begin{theorem}[Stability criterion: local maximizers]\label{bsc2}
 Let $1<p<\infty$ and $\bar v\in \mathring L^p(\mathbb S^2)$.  Denote $\mathcal R:=\mathcal R_{\bar v}$ and  $\mathcal C:=\mathcal C_{\bar v}$. Suppose that $\mathcal M\subset\mathcal R\cap\mathcal C$ satisfies
 \begin{itemize}
   \item [(1)] $\mathcal M$ is nonempty and compact in $L^p(\mathbb S^2);$
   \item [(2)] $\mathcal M$ is an isolated set of local maximizers of $E$ relative to $\mathcal R\cap \mathcal C$, i.e., there exists some open neighborhood $U$ of $\mathcal M$ in $L^p(\mathbb S^2)$  such that
       \begin{equation}\label{elv00}
E(u)= E(v):=E|_{\mathcal M} \quad\forall\,u, v\in\mathcal M,
       \end{equation}
       \begin{equation}\label{elv01}
E(u)\leq   E|_{\mathcal M} \quad\forall\,u\in \mathcal R\cap\mathcal C\cap  U,
       \end{equation}
       \begin{equation}\label{elv02}
u\in  {\mathcal R}\cap\mathcal C\cap U,\,\,  E(u)= E|_{\mathcal M} \quad \Longrightarrow \quad u\in\mathcal M.
\end{equation}
 \end{itemize}
  Then  $\mathcal M$  is $L^p$-stable   with respect to continuous admissible perturbations (cf. Definition \ref{dam01}), i.e., for any $\varepsilon>0,$ there exists some $\delta>0$, such that for any continuous  admissible map $\zeta(t)$, it holds that
     \begin{equation}\label{scm4}
     \min_{v\in\mathcal M }\|\zeta(0)-v\|_{L^p(\mathbb S^2)}<\delta\quad\Longrightarrow\quad  \min_{v\in\mathcal M }\|\zeta(t)-v\|_{L^p(\mathbb S^2)}<\varepsilon\quad\forall\,t\in\mathbb R.
     \end{equation}
  \end{theorem}

\subsection{Proof of Theorem \ref{bsc1}}

The following proposition is essential in proving Theorem \ref{bsc1}.

  \begin{proposition}\label{ps21}
 Let $1<p<\infty$ and $\zeta\in\mathring L^p(\mathbb S^2)$.  Then the following three statements  are equivalent:
 \begin{itemize}
   \item [(i)]  $\zeta$ is a local minimizer of $E$ relative to   ${\mathcal R}_\zeta\cap \mathcal C_\zeta$;
          \item [(ii)] $\zeta$ is a local minimizer of $E$ relative to $\bar{\mathcal R}_\zeta\cap \mathcal C_\zeta$;
   \item [(iii)]$\zeta$ is the unique minimizer of $E$ relative to   $\bar{\mathcal R}_\zeta\cap \mathcal C_\zeta$.
 \end{itemize}
\end{proposition}

\begin{proof}
  First we prove (i)$\Rightarrow$(ii).
Let $U$ be an open neighborhood of $\zeta$ in $L^p(\mathbb S^2)$ such that
       \begin{equation}\label{etg001}
E(\zeta)\leq   E(v) \quad\forall\,v\in \mathcal R_\zeta\cap\mathcal C_\zeta\cap  U.
       \end{equation}
  Since the weak and strong topologies of $L^p(\mathbb S^2)$ coincide on $\mathcal R_\zeta,$ there exists some weakly open set $V$ in $L^p(\mathbb S^2)$ such that $\mathcal R_\zeta\cap U=\mathcal R_\zeta\cap V.$ Hence by \eqref{etg001},
         \begin{equation}\label{etg002}
E(\zeta)\leq   E(v) \quad\forall\,v\in \mathcal R_\zeta\cap\mathcal C_\zeta\cap  V.
       \end{equation}
Now we claim that
\begin{equation}\label{zra01}
E(\zeta)\leq E(v) \quad \forall\,v\in \bar{\mathcal R}_\zeta\cap\mathcal C_\zeta\cap V.
\end{equation}
In fact, for any $v\in\bar{\mathcal R}_\zeta\cap\mathcal C_\zeta\cap V,$  we can choose a sequence $\{v_n\}\subset \mathcal R_\zeta\cap\mathcal C_\zeta$ such that $v_n\rightharpoonup v$ in $L^p(\mathbb S^2)$. Here and in the sequel, ``$\rightharpoonup$" denotes weak convergence. Since $V$ is weakly open,  it holds that  $v_n\in V$ if $n$ is sufficiently large. Therefore by \eqref{etg002} and the weak continuity of $E$,
\[E(\zeta)\leq \lim_{n\to\infty}E(v_n)=E(v),\]
which verifies \eqref{zra01}.

 Next we prove (ii)$\Rightarrow$(iii). Suppose that  \eqref{zra01} holds. Since $\bar{\mathcal R}_\zeta\cap\mathcal C_\zeta$ is weakly sequentially compact  and $E$ is weakly continuous, there exists a  minimizer of $E$ relative to $\bar{\mathcal R}_\zeta\cap\mathcal C_\zeta$, denoted by  $\xi$. Since
 $\bar{\mathcal R}_\zeta\cap\mathcal C_\zeta$ is a convex set and $E$ is a strictly convex functional,   $\xi$  must be the  unique minimizer of $E$ relative to   $\bar{\mathcal R}_\zeta\cap \mathcal C_\zeta$.
 In particular,
 \begin{equation}\label{tt01}
E(\zeta)\geq  E(\xi).
 \end{equation}
To prove (iii), it suffices to show that $\zeta=\xi$. Since $\bar{\mathcal R}_\zeta\cap\mathcal C_\zeta$ is convex, it holds that  $s\xi+(1-s)\zeta\in  \bar{\mathcal R}_\zeta \cap\mathcal C_\zeta$ for any  $s\in[0,1].$
Let $V$ be the open set in \eqref{zra01}. Fix a sufficiently small positive number $s$ such that
$s\xi+(1-s)\zeta\in  \bar{\mathcal R}_\zeta \cap\mathcal C_\zeta\cap V.$ Then by \eqref{zra01},
 \begin{equation}\label{tt02}
 E(s \xi+(1-s)\zeta)\geq E(\zeta).
 \end{equation}
On the other hand, due to the strict convexity of $E$,
 \begin{equation}\label{tt03}
 E(s \xi+(1-s)\zeta)\leq sE(\xi)+ (1-s)E(\zeta), \end{equation}
and the equality holds if and only if $\xi=\zeta.$
Combining \eqref{tt01}, \eqref{tt02} and \eqref{tt03}, we deduce that $E(\zeta)=E(\xi)$, and thus
$\zeta=\xi.$

Finally, the implication (iii)$\Rightarrow$(i) is obvious.
\end{proof}

With Proposition \ref{ps21} at hand,  we can prove the following compactness result.

    \begin{proposition}[Compactness]\label{cpct21}
 Let $\zeta$ be as in Theorem \ref{bsc1}. Then for any sequence $\{v_n\}\subset\mathcal R_\zeta$ satisfying
 \[\lim_{n\to\infty}E(v_n)=E(\zeta),\quad\lim_{n\to\infty}\mathbf m(v_n)=\mathbf m(\zeta),\]
 it holds that $v_n\to\zeta$ in $L^p(\mathbb S^2)$ as $n\to\infty.$
\end{proposition}
\begin{proof}
Since $\bar{\mathcal R}_\zeta$ is weakly sequentially compact in $L^p(\mathbb S^2)$, there exists some subsequence $\{v_{n_j}\}$ such that $v_{n_j}\rightharpoonup\xi$ in $L^p(\mathbb S^2)$ as $j\to\infty$ for some $\xi \in \bar{\mathcal R}_\zeta$. It is easy to check that
\begin{equation}\label{gx01}
\xi\in \bar{\mathcal R}_\zeta\cap\mathcal C_\zeta,\quad E(\xi)=E(\zeta).
\end{equation}
In fact, since $E$ and $\mathbf m$ are both weakly continuous in $L^p(\mathbb S^2),$ we have that
\[E(\xi)=\lim_{j\to\infty}E(v_{n_j})=E(\zeta),\quad \mathbf m(\xi)=\lim_{j\to\infty} \mathbf m(v_{n_j})= \mathbf m(\zeta).\]
On the other hand, by Proposition \ref{ps21}, $\zeta$ is the unique minimizer of $E$ relative to $\bar{\mathcal R}_\zeta\cap\mathcal C_\zeta$. So from \eqref{gx01}, we deduce that $\xi=\zeta$. To summarize, we have proved that the sequence $\{v_{n_j}\}$ satisfies
\[v_{n_j}\subset\mathcal R_\zeta,\quad v_{n_j}\rightharpoonup\zeta \quad \mbox{in}\quad L^p(\mathbb S^2).\]
Therefore $v_{n_j}\to\zeta$ in $L^p(\mathbb S^2)$ by uniform convexity.
  Furthermore, by a contradiction argument, we can prove that the whole sequence  $\{v_{n}\}$ converges to $\zeta$  in $L^p(\mathbb S^2)$.

\end{proof}

Now we are ready to prove Theorem \ref{bsc1}.

   \begin{proof}[Proof of Theorem \ref{bsc1}]
 It suffices to show that for any sequence of admissible maps $\{\zeta^n(t)\}$ and any sequence of times $\{t_n\}\subset\mathbb R$,  if
\begin{equation}\label{tt09}
\|\zeta^n(0)-\zeta\|_{L^p(\mathbb S^2)}\to 0,
\end{equation}
then
\begin{equation}\label{tt010}
\|\zeta^n(t_n)-\zeta\|_{L^p(\mathbb S^2)}\to 0.
\end{equation}
By \eqref{tt09},
 it is clear that
\[\lim_{n\to\infty}E(\zeta^n(0))= E(\zeta),\quad \lim_{n\to\infty}\mathbf m(\zeta^n(0))=\mathbf m(\zeta).\]
Taking into account the fact that $E$ and $\mathbf m$ are both conserved quantities for admissible maps, we have that
\begin{equation}\label{gx21}
\lim_{n\to\infty}E(\zeta^n(t_n))=\lim_{n\to\infty}E(\zeta^n(0))=E(\zeta),
\,\, \lim_{n\to\infty} \mathbf m(\zeta^n(t_n))=\lim_{n\to\infty}\mathbf m(\zeta^n(0))=\mathbf m(\zeta).
\end{equation}
Now, if  $\zeta^n(0)\in\mathcal R_\zeta$, then the desired stability is a straightforward consequence of \eqref{gx21} and Proposition \ref{cpct21}.
In fact, if $\zeta^n(0)\in\mathcal R_\zeta$, then $\zeta^n(t)\in\mathcal R_\zeta$ for any $t\in\mathbb R$ by the definition of admissible maps, and thus  Proposition \ref{cpct21} can be applied.
To deal with general perturbations, we need to construct a ``follower" $\eta_n\in\mathcal R_\zeta$ related to each $\zeta^n(t_n).$ More specifically, for each $n\in\mathbb Z_+$, we take some  $\eta_n\in\mathcal R_\zeta$ such that
\[\|\eta_n-\zeta^n(t_n)\|_{L^p(\mathbb S^2)}=\inf_{u\in\mathcal R_\zeta,v\in\mathcal R_{\zeta^n(0)}}\|u-v\|_{L^p(\mathbb S^2)}.\]
Note that such $\eta_n$ exists by \eqref{kd01}.
 In particular,
\[ \|\eta_n-\zeta^n(t_n)\|_{L^p(\mathbb S^2)}\leq \|\zeta-\zeta^n(0)\|_{L^p(\mathbb S^2)},\]
which in combination with  \eqref{tt09} yields
\begin{equation}\label{tt121}
\|\eta_n-\zeta^n(t_n)\|_{L^p(\mathbb S^2)} \to0\quad\mbox{as }\,\,n\to\infty.
\end{equation}
Hence
\begin{equation}\label{tt08}
\lim_{n\to\infty} E(\eta_n)= E(\zeta),\quad \lim_{n\to\infty}\mathbf m(\eta_n)=\mathbf m(\zeta).
\end{equation}
To summarize, we have constructed a sequence $\{\eta_n\}\subset \mathcal R_\zeta$ such that \eqref{tt08} holds.
Using Proposition \ref{cpct21}, we deduce that $\eta_n\to\zeta$ in $L^p(\mathbb S^2)$, which together with \eqref{tt121} leads to the desired assertion \eqref{tt010}.
   \end{proof}

  \subsection{Proof of Theorem \ref{bsc2}}

The proof of Theorem \ref{bsc2} is similar to that of Theorem \ref{bsc1}, with the major difference being the method used for verifying compactness.

We begin with the following lemma.

\begin{lemma}\label{rckd1}
Let $1<p<\infty$ and  ${\bar v}\in  L^p(\mathbb S^2)$. Denote $\mathcal R:=\mathcal R_{\bar v}$ and  $\mathcal C:=\mathcal C_{\bar v}$. Let $p'=p/(p-1)$ be the H\"older conjugate of $p$. For $w\in L^{p'}(\mathbb S^2),$ define a linear functional $L_w:L^{p}(\mathbb S^2)\mapsto \mathbb R$ by setting
\begin{equation}\label{lfde01}
L_w(v):=\int_{\mathbb S^2}wv\dd{\sigma},\quad v\in L^{p}(\mathbb S^2).
\end{equation}
Then there exists some $\eta\in \mathcal R\cap\mathcal C$ such that
\[L_w(\eta)=\sup_{v\in \bar{\mathcal R}\cap\mathcal C}L_w(v).\]
\end{lemma}
\begin{proof}

Since $\bar{\mathcal R}\cap\mathcal C$ is weakly sequentially compact and $L_w$ is weakly sequentially continuous, there exists some $\xi\in\bar{\mathcal R}\cap\mathcal C$ such that
\begin{equation}\label{zdd01}
L_w(\xi)=\sup_{v\in \bar{\mathcal R}\cap\mathcal C}L_w(v).
\end{equation}
Define
\[\Lambda:=\{v\in L^p(\mathbb S^2)\mid L_w(v)=L_w(\xi)\}.\]
It is easy to see that $\bar{\mathcal R}\cap\mathcal C\cap\Lambda$ is a weakly compact and convex set in $L^p(\mathbb S^2).$  By the   Krein–Milman theorem (cf. Chapter 3 of \cite{Rudin}), there is an extreme point of  $\bar{\mathcal R}\cap\mathcal C\cap\Lambda$, denoted by $\eta$.
 Since $\bar{\mathcal R}\cap\mathcal C$ is convex and  ${\mathcal R}\cap\mathcal C$ is the set of extreme points of $\bar{\mathcal R}\cap\mathcal C$, it suffices to show that $\eta$ is  an extreme point of $\bar{\mathcal R}\cap\mathcal C$. Suppose by contradiction that there exist  $\eta_1,\eta_2\in \bar{\mathcal R}\cap\mathcal C$ such that  $\eta_1,\eta_2\neq \eta$  and $\eta= (\eta_1+\eta_2)/2$. Then $\eta_1\notin \Lambda$ or $\eta_2\notin \Lambda$ since  $\eta$ is an extreme point of  $\bar{\mathcal R}\cap\mathcal C\cap\Lambda$. Hence
\begin{equation}\label{jcc1}
L_w(\eta_1)\neq L_w(\eta)\quad \mbox{or}\quad L_w(\eta_2)\neq L_w(\eta).
\end{equation}
On the other hand, since  $L_w$ is linear, it holds that
\begin{equation}\label{jcc2}
L_w(\eta)=\frac{1}{2}(L_w(\eta_1)+L_w(\eta_2)).
\end{equation}
From \eqref{jcc1} and \eqref{jcc2}, we deduce that $L_w(\eta_1)> L_w(\eta)$ or $L_w(\eta_2)> L_w(\eta)$, a contradiction to \eqref{zdd01}.

\end{proof}

\begin{lemma}\label{lem21}
 Let $1<p<\infty$ and  ${\bar v}\in \mathring L^p(\mathbb S^2)$. Denote $\mathcal R:=\mathcal R_{\bar v}$ and  $\mathcal C:=\mathcal C_{\bar v}$. Suppose  that $\zeta$ is a local maximizer of $E$ relative to $\bar{\mathcal R}\cap \mathcal C$, i.e., there exists  some open neighborhood  of $\zeta$ such that
\begin{equation}\label{leqt01}
  E(u)\leq E(\zeta)\quad \forall\,u\in\bar{\mathcal R}\cap \mathcal C\cap U.
\end{equation}
Then $\zeta\in\mathcal R.$
\end{lemma}
\begin{proof}
For  any $u\in\bar{\mathcal R}\cap\mathcal C$, since $\bar{\mathcal R}\cap \mathcal C$ is convex and $U$ is open, it holds that
\[s u+(1-s) \zeta\in\bar{\mathcal R}\cap\mathcal C\cap U\]
provided that $s$ is nonnegative and sufficiently small.
In view of \eqref{leqt01}, we have that
 \[\frac{d}{ds}E(s u+(1-s) \zeta)\bigg|_{s=0^+}\leq 0,\]
which implies
\begin{equation}\label{pq01}
\int_{\mathbb S^2}u\mathcal G \zeta \dd{\sigma}\leq \int_{\mathbb S^2} \zeta \mathcal G\zeta \dd{\sigma}.
\end{equation}
The above inequality holds for arbitrary $u\in\bar{\mathcal R}\cap\mathcal C$, therefore
 $\zeta$ is a maximizer of the linear functional $L_{\mathcal G \zeta}$  (cf. \eqref{lfde01})  relative to $\bar{\mathcal R}\cap\mathcal C$.
On the other hand, by Lemma \ref{rckd1}, the linear functional $L_{\mathcal G \zeta}$  must attain the maximum value relative to $\bar{\mathcal R}\cap\mathcal C$ at some  $\eta\in\mathcal R\cap\mathcal C$. In combination with \eqref{pq01}, one has
\begin{equation}\label{leqt02}
  \int_{\mathbb S^2}\eta\mathcal G\zeta \dd{\sigma}= \int_{\mathbb S^2}\zeta\mathcal G\zeta \dd{\sigma}.
\end{equation}
Fix a sufficiently small positive number $s$ such that
\[s\eta+(1-s)\zeta\in \bar{\mathcal R}\cap\mathcal C\cap U.\]
By \eqref{leqt01}, it holds that
\begin{equation}\label{lelt01}
E(s \eta+(1-s)\zeta)\leq E(\zeta).
\end{equation}
By a direct computation,  one has
\begin{equation}\label{lelt03}
E(s \eta+(1-s)\zeta)=E(\zeta)+s\int_{\mathbb S^2}(\eta-\zeta)\mathcal G\zeta \dd{\sigma}+ s^2E(\eta-\zeta).
\end{equation}
From \eqref{leqt02}-\eqref{lelt03}, we get $E(\eta-\zeta)\leq 0$.  Therefore  $\zeta=\eta\in\mathcal R$ by the positive-definiteness of $\mathcal G$ (cf. \eqref{podf01}).

\end{proof}

 \begin{lemma}\label{cxv01}
 Let $\mathcal M$ be as in Theorem \ref{bsc2}.  Then there exists some weakly open set $V\subset L^p(\mathbb S^2)$ such that $\mathcal M$ is exactly the set of maximizers of $E$ relative to $\bar{\mathcal R}\cap \mathcal C\cap V$.
 \end{lemma}
 \begin{proof}

Let $U$ be the open set as in Theorem \ref{bsc2}.
As in the proof of Proposition \ref{ps21}, we can choose some weakly open set $V\subset L^p(\mathbb S^2)$ such that $\mathcal R\cap U=\mathcal R\cap V.$ By \eqref{elv01} and \eqref{elv02},
\begin{equation}\label{peq054}
E(u)\leq E|_{\mathcal M}\quad \forall\,u\in  {\mathcal R}\cap\mathcal C\cap V.
\end{equation}
 \begin{equation}\label{peq055}
u\in  {\mathcal R}\cap\mathcal C\cap V,\,\,  E(u)= E|_{\mathcal M} \quad \Longrightarrow \quad u\in\mathcal M.
\end{equation}
To complete the proof, it suffices to show that
\begin{equation}\label{cdd01}
E(u)\leq E|_{\mathcal M}\quad \forall\,u\in \bar{\mathcal R}\cap\mathcal C\cap V,
\end{equation}
\begin{equation}\label{cdd02}
u\in \bar{\mathcal R}\cap\mathcal C\cap V,\,\,  E(u)= E|_{\mathcal M} \quad \Longrightarrow \quad u\in\mathcal M.
\end{equation}
To prove \eqref{cdd01}, notice that for any $u\in\bar{\mathcal R}\cap\mathcal C\cap V,$ there exists a sequence  $\{u_n\}\subset \mathcal R\cap\mathcal C$ such that $u_n\rightharpoonup u$ in $L^p(\mathbb S^2)$. Since $V$ is weakly open, it holds that $u_n\in V$ if $n$ is large enough. So by \eqref{peq054},
\[E(u)=\lim_{n\to+\infty}E(u_n)\leq E|_{\mathcal M}.\]
Hence \eqref{cdd01} has been verified.
Next we prove \eqref{cdd02}.
Suppose that $u\in \bar{\mathcal R}\cap\mathcal C\cap V$  satisfies $E(u)= E|_{\mathcal M}$, then $u$ must be a local maximizer of $E$ relative to $\bar{\mathcal R}\cap \mathcal C$ by \eqref{cdd01}. Using Lemma \ref{lem21}, we get  $u\in\mathcal R$. Therefore $u\in\mathcal M$ by \eqref{peq055}.

 \end{proof}

     \begin{proposition}[Compactness]\label{cpct22}
 Let $\mathcal M$ be as in Theorem \ref{bsc2}. Then there exists some $\tau>0,$ such that for any sequence $\{v_n\}\subset\mathcal R\cap\mathcal C$ satisfying
 \begin{equation}\label{peqas1}
 \min_{v\in\mathcal M}\|v_n-v\|_{L^p(\mathbb S^2)}<\tau\quad\forall\,n\in\mathbb Z_+,\quad \lim_{n\to\infty}E(v_n)=E|_{\mathcal M},
 \end{equation}
 there exist a subsequence $\{v_{n_j}\}$ and  $\eta\in\mathcal M$ such that
  $v_{n_j}\to\eta$ in $L^p(\mathbb S^2)$ as $j\to\infty.$
\end{proposition}

  \begin{proof}
Let $V$ be as in Lemma \ref{cxv01}.
Since $\mathcal M$ is compact, we can choose a $\tau >0$ such that
\[\left\{u\in L^p(\mathbb S^2)\mid \min_{v\in\mathcal M}\|u-v\|_{L^p(\mathbb S^2)}\leq \tau\right\}\subset V.\]
Let  $\{v_n\}\subset\mathcal R\cap\mathcal C$ be a sequence satisfying \eqref{peqas1}. Since $\{v_n\}$ is obviously bounded in $L^p(\mathbb S^2)$, there exist a subsequence  $\{v_{n_j}\}$ and some $\eta\in\bar{\mathcal R}\cap\mathcal C$ such that $v_{n_j}\rightharpoonup \eta$  in $L^p(\mathbb S^2)$ as $j\to\infty$.  It is clear that $E(\eta)=E|_{\mathcal M}$. On the other hand, since \[\min_{v\in\mathcal M}\|\eta-v\|_{L^p(\mathbb S^2)}\leq \tau,\]
we have that  $\eta\in V$.
To conclude,  $\eta$ is a maximizer of $E$ relative to $\bar{\mathcal R}\cap \mathcal C\cap V$, and thus $\eta\in\mathcal M$ by Lemma \ref{cxv01}. Besides, $v_{n_j}\to \eta$ in $L^p(\mathbb S^2)$ as $j\to\infty$ by uniform convexity.

   \end{proof}

Having made the necessary preparations, we are ready to prove Theorem \ref{bsc2}.
   \begin{proof}[Proof of Theorem \ref{bsc2}]
 Suppose by contradiction that there exist some $\varepsilon_0>0$, some sequence of continuous admissible maps $\{\zeta^n(t)\}$, and some sequence of times $\{t_n\}\subset\mathbb R$ such that
\begin{equation}\label{pp01}
\min_{v\in\mathcal M}\|\zeta^n(0)-v\|_{L^p(\mathbb S^2)}\to0 \quad \mbox{as}\quad n\to\infty,
\end{equation}
\begin{equation}\label{pp0101}
\min_{v\in\mathcal M}\|\zeta^n(t_n)-v\|_{L^p(\mathbb S^2)}\geq \varepsilon_0\quad \forall\,n\in\mathbb Z_+.
\end{equation}
Without loss of generality, we assume that $\varepsilon_0<\tau$, where $\tau$ is the positive number determined in Proposition \ref{cpct22}. Since each $\zeta^n(t)$ is continuous,
the following function
\[\min_{v\in\mathcal M}\|\zeta^n(t)-v\|_{L^p(\mathbb S^2)}\]
is  continuous with respect to $t$. Hence by reselecting a new sequence of times, still denoted by $\{t_n\}$, we can assume that
\begin{equation}\label{pp02}
\min_{v\in\mathcal M}\|\zeta^n(t_n)-v\|_{L^p(\mathbb S^2)}= \varepsilon_0<\tau\quad \forall\,n\in\mathbb Z_+.
\end{equation}
As in the proof of Theorem \ref{bsc1}, for each $n\in\mathbb Z_+$,
we can  take some ``follower" $\eta_n\in\mathcal R$  such that
\[\|\zeta^n(t_n)-\eta_n\|_{L^p(\mathbb S^2)}=\min_{u\in \mathcal R_{\zeta^n(t_n)},\,v\in\mathcal R}\|u-v\|_{L^p(\mathbb S^2)}.\]
Since $\mathcal R_{\zeta^n(t_n)}=\mathcal R_{\zeta^n(0)}$, we have that
\begin{equation}\label{pp03}
\|\zeta^n(t_n)-\eta_n\|_{L^p(\mathbb S^2)}\leq \|\zeta^n(0)-v\|_{L^p(\mathbb S^2)}\quad\forall\,v\in\mathcal R,\,n\in\mathbb Z_+.
\end{equation}
By \eqref{pp01} and \eqref{pp03},
\begin{equation}\label{pp04}
\|\zeta^n(t_n)-\eta_n\|_{L^p(\mathbb S^2)}\to 0\,\,\,\,\mbox{as}\,\,\,\, n\to\infty.
\end{equation}
 Taking into account \eqref{pp02}, we deduce that
\begin{equation}\label{pp041}
\min_{v\in\mathcal M}\|\eta_n-v\|_{L^p(\mathbb S^2)} <\tau\quad\mbox{for sufficiently large $n$.}
\end{equation}
On the other hand, by \eqref{pp01},
\[\lim_{n\to+\infty} E(\zeta^n(0))=E|_{\mathcal M},\]
which, in combination with the fact that $E$ is conserved for admissible maps,  yields
\begin{equation}\label{pp05}
\lim_{n\to\infty} E(\zeta^n(t_n))=E|_{\mathcal M}.
\end{equation}
From \eqref{pp04} and \eqref{pp05}, we deduce that
\begin{equation}\label{pp051}
\lim_{n\to\infty} E(\eta_n)=E|_{\mathcal M}.
\end{equation}
To summarize, we have constructed a sequence $\{\eta_n\}\subset\mathcal R$ such that \eqref{pp041} and \eqref{pp051} hold. By Proposition \ref{cpct22}, we deduce that, up to a subsequence, $\eta_n\to\zeta$ in $L^p(\mathbb S^2)$ for some $\zeta\in\mathcal M.$ This is a contradiction to \eqref{pp02} and \eqref{pp04}.
   \end{proof}

\begin{remark}
In the proofs of Theorems \ref{bsc1} and \ref{bsc2}, we develop a new approach to constructing the ``follower" $\eta_n$, which is different from the one in \cite{BAR}  by solving a linear transport equation. The advantage of our approach  is that it does not require additional discussion on the regularity of perturbed solutions.
\end{remark}

\section{Proofs of stability theorems}\label{sc4}

In this section, we give the proofs of Theorems \ref{a1s}, \ref{a2s2} and \ref{a2s3}. We will show that the solutions in  Theorem  \ref{a1s} satisfy the conditions of Theorem \ref{bsc1}, and the solutions in  Theorems \ref{a2s2} and \ref{a2s3} satisfy the conditions of Theorem \ref{bsc2}.

\subsection{Proof of Theorem \ref{a1s}}

According to Theorem \ref{bsc1}, Theorem \ref{a1s}  is a straightforward consequence of the following proposition.
\begin{proposition}\label{ps31}
Let $\zeta$ be as in Theorem \ref{a1s}. Then $\zeta$ is the unique minimizer of $E$ relative to ${\mathcal R}_\zeta\cap\mathcal C_\zeta.$
\end{proposition}
\begin{proof}

 We will show that $\zeta$ is the unique minimizer of $E$ relative to $\bar{\mathcal R}_\zeta\cap\mathcal C_\zeta$.
 Let $\xi$ be the unique minimizer of $E$ relative to $\bar{\mathcal R}_\zeta\cap\mathcal C_\zeta$. Note that such $\xi$ exists since $\bar{\mathcal R}_\zeta\cap\mathcal C_\zeta$ is weakly sequentially compact and convex,  and $E$ is weakly continuous and strictly convex. Then
  \begin{equation}\label{twi1}
E(\zeta)\geq  E(\xi).
  \end{equation}
 According to  \cite[Lemma 3]{BMA},  the condition \eqref{nav1} in Theorem \ref{a1s} implies that
 $\zeta$ is a minimizer of the  linear functional $L_{\mathcal G\zeta+\beta\mathbf p\cdot\mathbf x}$ (cf. \eqref{lfde01})  relative to $\mathcal R_\zeta$.
Taking into account the fact that $ <\beta\mathbf p\cdot\mathbf x,\cdot>$ is constant on $\mathcal C_\zeta$, we deduce that
 $\zeta$ is a minimizer of $L_{ \mathcal G\zeta}$ relative to $\mathcal R_\zeta\cap\mathcal C_\zeta$, and thus
 $\zeta$ is a minimizer of $L_{ \mathcal G\zeta}$ relative to $\bar{\mathcal R}_\zeta\cap\mathcal C_\zeta$.
In particular,
  \begin{equation}\label{twi2}
  \int_{\mathbb S^2}\zeta \mathcal G\zeta \dd{\sigma} \leq \int_{\mathbb S^2}\xi\mathcal G\zeta \dd{\sigma}.
  \end{equation}
 On the other hand, we have the following identity:
   \begin{equation}\label{twi21}
   E(\xi)= E(\zeta)+\int_{\mathbb S^2}\mathcal G\zeta(\xi-\zeta)\dd{\sigma}+E(\xi-\zeta).
   \end{equation}
In view of \eqref{twi1}, \eqref{twi2}  and \eqref{twi21}, we obtain $E(\xi-\zeta)\leq 0$. Therefore $\zeta=\xi$ by the positive-definiteness of $\mathcal G.$ This completes the proof.

\end{proof}

\subsection{Proof  of Theorem   \ref{a2s2}}

By Theorem \ref{bsc2}, the stability assertion in Theorem   \ref{a2s2} follows from the obvious fact that $(\zeta+\mathbb E_2)\cap \mathcal R_\zeta$ is compact in $L^p(\mathbb S^2)$ and the following variational characterization.

\begin{proposition}\label{ps42}
Let $\zeta$ be as in Theorem \ref{a2s2}. Then $(\zeta+\mathbb E_2)\cap \mathcal R_\zeta$ is the set of maximizers of $E$ relative to ${\mathcal R}_\zeta\cap\mathcal C_\zeta.$
\end{proposition}
\begin{proof}
Define $G(s)=\int_0^sg(\tau)\dd{\tau}$. Let $\hat G$ be the Legendre transform of $G$, i.e.,
\begin{equation}\label{ltsfm1}
\hat {G}(s):=\sup_{\tau\in\mathbb R}(s\tau-{G}(\tau)),\quad s\in\mathbb R.
\end{equation}
Note that to make $\hat {G}$ well-defined, one may need to redefine the values of $g$ outside the interval $[m_1,m_2]$ such that $g(s)$ grows linearly as $|s|\to\infty$, where
\[m_1:=\min_{\mathbf x\in\mathbb S^2}(\mathcal G\zeta(\mathbf x)+\beta\mathbf p\cdot\mathbf x),\quad m_2:=\max_{\mathbf x\in\mathbb S^2}(\mathcal G\zeta(\mathbf x)+\beta\mathbf p\cdot\mathbf x).\]
 For example, we can redefine $g$   on $(-\infty,m_1]$ as follows (the case of $[m_2,+\infty)$ can be handled similarly):
\begin{itemize}
  \item [(i)]If $g'(m_1)>0$,  define
\begin{equation}
g(s)=
g'(m_1)(s-m_1)+g(m_1),\quad s\in(-\infty,m_1).
\end{equation}
  \item [(ii)]If $g'(m_1)=0$,  define
\begin{equation}
g(s)=
\begin{cases}
g(m_1)-(s-m_1)^2&\mbox{ if } m_1-1\leq s< m_1,\\
2(s-m_1+1)+g(m_1)-1 &\mbox{ if } s< m_1-1.
\end{cases}
\end{equation}
\end{itemize}
From the definition of Legendre transform, one has
\[\hat {G}(s)+  G (\tau)\geq s\tau \quad\forall\,s,\tau\in\mathbb R,\]
and the equality holds if and only if $s= g(\tau).$ In particular
\[\hat {G}(\zeta )+  G  (\mathcal G \zeta +\beta\mathbf p\cdot\mathbf x)=  \zeta (\mathcal G \zeta+\beta\mathbf p\cdot\mathbf x).\]
Besides, one can verify that
 $\hat{G}$ is locally Lipschitz continuous.

Fix a function $\varrho$ such that $ \zeta+\varrho\in\mathcal R_{\zeta}\cap \mathcal C_\zeta$.  Then it is easy to see that
\begin{equation}\label{mp00}
\mathbf m(\varrho) =\mathbf 0.
\end{equation}
 We compute as follows:
    \begin{equation}\label{js19}
    \begin{split}
  & E( \zeta)- E( \zeta+\varrho)\\
  =&\frac{1}{2}\int_{\mathbb S^2} \left(\zeta\mathcal G \zeta -( \zeta+\varrho)\mathcal G ( \zeta+\varrho)\right) \dd{\sigma}+\int_{\mathbb S^2}\left( \hat {G}( \zeta+\varrho)-\hat {G}( \zeta)\right) \dd{\sigma} \\
 \geq &\frac{1}{2}\int_{\mathbb S^2} \left(- 2\varrho\mathcal G \zeta-\varrho\mathcal G\varrho \right)\dd{\sigma}+\int_{\mathbb S^2}\left(( \zeta+\varrho)(\mathcal G( \zeta+\varrho)+\beta\mathbf p\cdot\mathbf x)- G(\mathcal G( \zeta+\varrho)+\beta\mathbf p\cdot\mathbf x) \right)\dd{\sigma}\\
 &-\int_{\mathbb S^2}\left(\zeta(\mathcal G\zeta+\beta\mathbf p\cdot\mathbf x)-G(\mathcal G\zeta+\beta\mathbf p\cdot\mathbf x)\right) \dd{\sigma} \\
=& \frac{1}{2}\int_{\mathbb S^2}\varrho\mathcal G\varrho \dd{\sigma}+\int_{\mathbb S^2}\varrho\mathcal G\zeta \dd{\sigma}-\int_{\mathbb S^2} \left(G(\mathcal G(\zeta+\varrho)+\beta\mathbf p\cdot\mathbf x)-G(\mathcal G\zeta+\beta\mathbf p\cdot\mathbf x) \right)\dd{\sigma}.
\end{split}
\end{equation}
Since $  g'\leq 6$, we  have that
\[ G(s+\tau)\leq  G(s)+ g(s) \tau+3\tau^2,\quad\forall\,s,\tau\in\mathbb R.\]
Therefore
\begin{equation}\label{js20}
\begin{split}
&\int_{\mathbb S^2}   \left(G(\mathcal G(\zeta+\varrho)+\beta\mathbf p\cdot\mathbf x)-  G(\mathcal G\zeta+\beta\mathbf p\cdot\mathbf x) \right)\dd{\sigma}\\
 \leq& \int_{\mathbb S^2}   g(\mathcal G \zeta +\beta\mathbf p\cdot\mathbf x)\mathcal G\varrho \dd{\sigma}+3\int_{\mathbb S^2}(\mathcal G\varrho)^2 \dd{\sigma}\\
 =&\int_{\mathbb S^2} \zeta\mathcal G\varrho \dd{\sigma}+3\int_{\mathbb S^2}(\mathcal G\varrho)^2 \dd{\sigma}.
 \end{split}
\end{equation}
Inserting \eqref{js20} into \eqref{js19}, we have that
\[
 E( \zeta)- E( \zeta+\varrho)\geq \frac{1}{2}\int_{\mathbb S^2} \varrho\mathcal G\varrho \dd{\sigma}-3\int_{\mathbb S^2}(\mathcal G\varrho)^2 \dd{\sigma}.
 \]
In view of \eqref{mp00},  the following Poincar\'e inequality  holds:
\begin{equation}\label{js21}
\int_{\mathbb S^2}(\mathcal G \varrho)^2 \dd{\sigma}\leq \frac{1}{6}\int_{\mathbb S^2} \varrho\mathcal G\varrho \dd{\sigma}.
\end{equation}
So we obtain
$
 E( \zeta)\geq E( \zeta+\varrho),
$
and the equality holds if and only if \eqref{js19}-\eqref{js21} are all equalities.

Below we analyze \eqref{js19}-\eqref{js21}.
It is clear that \eqref{js21} is an equality if and only if  $\varrho\in\mathbb E_2.$ Now we claim that
\begin{equation}\label{clm01}
\mbox{if $\varrho\in\mathbb E_2,$ then \eqref{js19} and \eqref{js20} are  both  equalities.}
\end{equation}
In fact, since $\zeta+\varrho\in\mathcal R_{\zeta}$, one has
\[ \int_{\mathbb S^2}\hat {G}( \zeta+\varrho)\dd{\sigma}=\int_{\mathbb S^2}\hat {G}( \zeta)\dd{\sigma}.\]
On the other hand, in view of \eqref{js20},
\begin{align*}
  &\int_{\mathbb S^2}\left(\hat {G}( \zeta+\varrho)-\hat {G}( \zeta)\right)\dd{\sigma} \\
 \geq & \int_{\mathbb S^2}\left(( \zeta+\varrho)(\mathcal G( \zeta+\varrho)+\beta\mathbf p\cdot\mathbf x)-  G(\mathcal G( \zeta+\varrho)+\beta\mathbf p\cdot\mathbf x) \right)\dd{\sigma}\\
 &- \int_{\mathbb S^2}\left(\zeta(\mathcal G\zeta+\beta\mathbf p\cdot\mathbf x)- G(\mathcal G\zeta+\beta\mathbf p\cdot\mathbf x) \right)\dd{\sigma}\\
= &\int_{\mathbb S^2}\left(( \zeta+\varrho) \mathcal G( \zeta+\varrho) -\zeta \mathcal G\zeta\right)  \dd{\sigma}-\int_{\mathbb S^2} \left(G(\mathcal G( \zeta+\varrho)+\beta\mathbf p\cdot\mathbf x)- G(\mathcal G\zeta+\beta\mathbf p\cdot\mathbf x) \right)\dd{\sigma}\\
\ge&\int_{\mathbb S^2}\left(( \zeta+\varrho) \mathcal G( \zeta+\varrho) -\zeta \mathcal G\zeta\right)  \dd{\sigma}-\int_{\mathbb S^2} \zeta\mathcal G\varrho \dd{\sigma}+3\int_{\mathbb S^2}(\mathcal G\varrho)^2 \dd{\sigma}\\
=&\int_{\mathbb S^2} \left(\varrho  \mathcal G\varrho +\zeta\mathcal G\varrho\right) \dd{\sigma} -3\int_{\mathbb S^2}(\mathcal G\varrho)^2 \dd{\sigma}\\
=&\frac{1}{6}\int_{\mathbb S^2} \left(\varrho^2 +\zeta \varrho\right) \dd{\sigma}-\frac{1}{12}\int_{\mathbb S^2} \varrho^2   \dd{\sigma}\\
=&\frac{1}{12}\int_{\mathbb S^2}\varrho^2 \dd{\sigma}+\frac{1}{6}\int_{\mathbb S^2} \zeta \varrho \dd{\sigma}\\
=&0.
 \end{align*}
 Note that the last equality follows from the fact that  $\|\zeta+\varrho\|_{L^2(\mathbb S^2)}=\|\zeta\|_{L^2(\mathbb S^2)}$.
So \eqref{clm01} has been verified.

To summarize, we have proved that $\zeta+\varrho\in\mathcal R_\zeta\cap\mathcal C_\zeta$ is  a maximizer of $E$ relative to $\mathcal R_\zeta\cap\mathcal C_\zeta$ if and only if  $\varrho\in\mathbb E_2.$ In other words,  the set of maximizers of $E$ relative to $\mathcal R_\zeta\cap\mathcal C_\zeta$ is exactly  $(\zeta+\mathbb E_2)\cap \mathcal R_\zeta\cap\mathcal C_\zeta$. Notice that
\[(\zeta+\mathbb E_2)\cap \mathcal R_\zeta\cap\mathcal C_\zeta=(\zeta+\mathbb E_2)\cap \mathcal R_\zeta,\]
 since $\zeta+\varrho\in \mathcal C_\zeta$ holds automatically  for any $\varrho\in\mathbb E_2$. This completes the proof.
\end{proof}

Next we show that the set $(\zeta+\mathbb E_2)\cap \mathcal R_\zeta$ is a singleton if
$\sup_{\mathbb R}g' <6.$
\begin{proposition}
In the setting of Theorem \ref{a2s2}, if  additionally
$\sup_{\mathbb R}g' <6,$
then
\[(\zeta+\mathbb E_2)\cap \mathcal R_\zeta=\{\zeta\}.\]
\end{proposition}
\begin{proof}
For any  $\zeta+\varrho\in (\zeta+\mathbb E_2)\cap \mathcal R_\zeta,$ by \eqref{clm01} we know that the inequality in \eqref{js19} is an equality, i.e.,
\[\int_{\mathbb S^2} \hat G(\zeta+\varrho)\dd{\sigma}=\int_{\mathbb S^2}\left(( \zeta+\varrho)(\mathcal G( \zeta+\varrho)+\beta\mathbf p\cdot\mathbf x)- G(\mathcal G( \zeta+\varrho)+\beta\mathbf p\cdot\mathbf x)\right) \dd{\sigma},\]
which implies that
\[\hat G(\zeta+\varrho)=( \zeta+\varrho)(\mathcal G( \zeta+\varrho)+\beta\mathbf p\cdot\mathbf x)- G(\mathcal G( \zeta+\varrho)+\beta\mathbf p\cdot\mathbf x).\]
Hence $\varrho$ satisfies
\begin{equation}\label{sfyse1}
\zeta+\varrho=g(\mathcal G( \zeta+\varrho)+\beta\mathbf p\cdot\mathbf x).
\end{equation}
Taking into account  $\zeta= g(\mathcal G\zeta+\beta\mathbf p\cdot\mathbf x)$, we have that
\[\varrho= g(\mathcal G( \zeta+\varrho)+\beta\mathbf p\cdot\mathbf x)- g(\mathcal G \zeta  +\beta\mathbf p\cdot\mathbf x).\]
Therefore
\[|\varrho|=| g(\mathcal G(\zeta+\varrho)+\beta\mathbf p\cdot\mathbf x)- g(\mathcal G\zeta+\beta\mathbf p\cdot\mathbf x)|\leq  \sup_{\mathbb R} g' |\mathcal G\varrho|=\frac{ \sup_{\mathbb R} g'}{6}|\varrho|,\]
which yields $\varrho\equiv 0$ provided that $\sup_{\mathbb R}g'<6.$

\end{proof}

\subsection{Proof  of Theorem   \ref{a2s3}}

By Theorem \ref{bsc2}, we only need to prove the following  proposition.

\begin{proposition}\label{ps43}
Let $\zeta$ be as in Theorem \ref{a2s3}. Then $\zeta$ is an isolated local maximizer of $E$ relative to ${\mathcal R}_\zeta\cap\mathcal C_\zeta,$ i.e., there exists some $\tau>0$, such that for any $v\in\mathcal R_\zeta\cap\mathcal C_\zeta$ with
\[\|v-\zeta\|_{L^p(\mathbb S^2)}<\tau,\quad v\neq \zeta,\] it holds that
\[E(v)<E(\zeta).\]

\end{proposition}

\begin{proof}
 For $ \zeta+\varrho\in\mathcal R_{\zeta}\cap \mathcal C_\zeta$, repeating the computations  in \eqref{js19}, we have that
     \begin{equation}\label{js69}
    \begin{split}
  & E( \zeta)- E( \zeta+\varrho)\\
\ge & \frac{1}{2}\int_{\mathbb S^2}\varrho\mathcal G\varrho \dd{\sigma}+\int_{\mathbb S^2}\varrho\mathcal G\zeta \dd{\sigma}-\int_{\mathbb S^2}  \left( G(\mathcal G(\zeta+\varrho)+\beta\mathbf p\cdot\mathbf x)-  G(\mathcal G\zeta+\beta\mathbf p\cdot\mathbf x) \right)\dd{\sigma}.
\end{split}
\end{equation}
By Taylor's theorem,
\begin{equation}\label{js70}
\begin{split}
 &\int_{\mathbb S^2}
  \left(G(\mathcal G(\zeta+\varrho)+\beta\mathbf p\cdot\mathbf x)-  G(\mathcal G\zeta+\beta\mathbf p\cdot\mathbf x) \right)\dd{\sigma}\\
  =&\int_{\mathbb S^2} g(\mathcal G\zeta+\beta\mathbf p\cdot\mathbf x)\mathcal G\varrho \dd{\sigma}+\frac{1}{2}\int_{\mathbb S^2} g'(\mathcal G\zeta+\beta\mathbf p\cdot\mathbf x+\tau\mathcal G\varrho)(\mathcal G\varrho)^2 \dd{\sigma}\\
  =&\int_{\mathbb S^2}\zeta\mathcal G\varrho \dd{\sigma}+\frac{1}{2}\int_{\mathbb S^2}  g'(\mathcal G\zeta+\beta\mathbf p\cdot\mathbf x+\tau\mathcal G\varrho)(\mathcal G\varrho)^2 \dd{\sigma},
 \end{split}
\end{equation}
where $\tau\in (0,1)$ depends on $\varrho.$
Inserting \eqref{js70} into \eqref{js69}, we have that
\begin{equation}\label{pe101}
 E( \zeta)- E( \zeta+\varrho)\geq \frac{1}{2}\int_{\mathbb S^2}\varrho\mathcal G\varrho \dd{\sigma}-\frac{1}{2}\int_{\mathbb S^2} g'(\mathcal G\zeta+ \beta\mathbf p\cdot\mathbf x+\tau\mathcal G\varrho)(\mathcal G\varrho)^2 \dd{\sigma}.
\end{equation}
Since $\mathcal G$ is  a compact operator from $\mathring L^p(\mathbb S^2)$ into $\mathring L^\infty(\mathbb S^2)$, one has
\[  \|\mathcal G\varrho\|_{L^\infty(\mathbb S^2)}\to 0\quad \mbox{as}\quad \|\varrho\|_{L^p(\mathbb S^2)}\to 0.\]
So there exists some  $r>0$ such that
 \begin{equation}\label{pe102}
 \|g'(\mathcal G\zeta+ \beta\mathbf p\cdot\mathbf x+\tau\mathcal G\varrho)-g'(\mathcal G\zeta+ \beta\mathbf p\cdot\mathbf x)\|_{L^\infty(\mathbb S^2)}<\frac{c}{2}
 \end{equation}
 as long as $\|\varrho\|_p<r$, where $c$ is the positive number in \eqref{pe1}.
 From \eqref{pe101} and \eqref{pe102}, we obtain
 \begin{equation}\label{pe103}
 E( \zeta)- E( \zeta+\varrho)\geq  \frac{c}{4}\int_{\mathbb S^2}(\mathcal G\varrho)^2 \dd{\sigma}>0
\end{equation}
 if $0<\|\varrho\|_{L^p(\mathbb S^2)}<r$. This completes the proof.
\end{proof}

\section{Proof of rigidity theorem}\label{sc5}

In this section, we give the proof of Theorem \ref{rgr} based on the variational characterizations in Propositions \ref{ps31}, \ref{ps42} and \ref{ps43}.

\begin{proof}[Proof of Theorem \ref{rgr}(i)]
 If $\mathbf m(\zeta)=\mathbf 0,$ then for any rigid rotation $\mathsf R\in\mathbb S\mathbb O(3),$ one has
 \begin{equation}\label{skl1}
 \mathbf m(\zeta\circ\mathsf R)=\int_{\mathbb S^2}\mathbf x\zeta(\mathsf R\mathbf x)\dd{\sigma}=\mathsf R^{-1}\int_{\mathbb S^2}\mathbf x\zeta(\mathbf x)\dd{\sigma}=\mathbf 0,
 \end{equation}
 which means $\zeta\circ\mathsf R\in\mathcal C_\zeta.$ Moreover, it is clear that \[\zeta\circ\mathsf R\in\mathcal R_\zeta,\quad E(\zeta\circ\mathsf R)=E(\zeta).\]
  So $\zeta\circ\mathsf R$ is a minimizer of $E$ relative to $\mathcal R_\zeta\cap\mathcal C_\zeta$. By uniqueness of the minimizer in Proposition \ref{ps31}, we deduce that $\zeta\circ\mathsf R=\zeta.$ Since $\mathsf R\in\mathbb S\mathbb O(3)$ is arbitrary, we must have $\zeta\equiv 0.$ So Theorem \ref{rgr}(i) holds in this case.

Below we assume that
\begin{equation}\label{qdp1}
\mathbf m(\zeta)\neq\mathbf 0.
\end{equation}
Define
 \begin{equation}\label{qd0}
 \mathbf q:=\frac{\mathbf m(\zeta)}{|\mathbf m(\zeta)|}.
 \end{equation}
 Then $\zeta\circ\mathsf R^{\mathbf q}_\theta\in\mathcal C_\zeta$ for any $\theta\in\mathbb R$. In fact,
  \[\mathsf m(\zeta\circ\mathsf R^{\mathbf q}_\theta)=\int_{\mathbb S^2}\mathbf x\zeta(\mathsf R^{\mathbf q}_{\theta}\mathbf x)\dd{\sigma}=\mathsf R^{\mathbf q}_{-\theta}\int_{\mathbb S^2}\mathbf x\zeta(\mathbf x)\dd{\sigma}=\mathsf R^{\mathbf q}_{-\theta}\mathbf m(\zeta)=\mathbf m(\zeta).\]
 Besides, it is clear that
\[\zeta\circ\mathsf R^{\mathbf q}_\theta\in\mathcal R_\zeta,\quad E(\zeta\circ\mathsf R^{\mathbf q}_\theta)=E(\zeta).\]
  So   $\zeta\circ\mathsf R^{\mathbf q}_\theta$ is a minimizer of
   $E$ relative to $\mathcal R_\zeta\cap\mathcal C_\zeta$ for any $\theta\in\mathbb R$. By uniqueness of the minimizer again,  $\zeta$ must be rotationally invariant around $\mathbf q,$ i.e.,
 \begin{equation}\label{rzta1}
 \zeta=\zeta\circ\mathsf R^{\mathbf q}_\theta  \quad\forall\, \theta\in\mathbb R,
 \end{equation}
  which in combination with \eqref{nav1} implies that
 \begin{equation}\label{rzta2}
 g(\mathcal G\zeta+\beta\mathbf p\cdot\mathbf x)=g(\mathcal G\zeta+\beta\mathbf p\cdot(\mathsf R^{\mathbf q}_\theta\mathbf x))\quad\forall\,\theta\in\mathbb R.
 \end{equation}
  To finish the proof, it suffices to show that $\mathbf q=\pm \mathbf p$ under the condition $\beta\neq 0.$
Suppose by contradiction that $\mathbf q\neq\pm \mathbf p$.
For any $\mathbf x\in\mathbb S^2$, the following Rodrigues's rotation formula holds:
\begin{equation}\label{rrf0}
\mathbf p\cdot(\mathsf R^{\mathbf q}_\theta\mathbf x)=(\cos\theta) \mathbf p\cdot\mathbf x +\sin\theta\mathbf p\cdot (\mathbf q\times \mathbf x)+(1-\cos\theta)(\mathbf q\cdot\mathbf x)(\mathbf p\cdot\mathbf q).
\end{equation}
If $\mathbf p\cdot (\mathbf q\times \mathbf x)\neq 0,$
then it is easy to check that
\begin{equation}\label{fft1}
 \max_{\theta\in\mathbb R}[\beta\mathbf p\cdot(\mathsf R^{\mathbf q}_\theta\mathbf x)-\beta\mathbf p\cdot\mathbf x]>0,\quad \min_{\theta\in\mathbb R}[\beta\mathbf p\cdot(\mathsf R^{\mathbf q}_\theta\mathbf x)-\beta\mathbf p\cdot\mathbf x]<0,
\end{equation}
which in combination with  \eqref{rzta2} implies that  $g(\mathcal G\zeta+\beta\mathbf p\cdot\mathbf x)$ remains constant for a sufficiently small change in the value of $\mathbf x$. In other words,
 $g(\mathcal G\zeta+\beta\mathbf p\cdot\mathbf x)$, and thus $\zeta,$ is constant in some neighborhood of $\mathbf x$ as long as $\mathbf p\cdot (\mathbf q\times \mathbf x)\neq 0.$ As a result, $\zeta$ must be constant in the  following two connected components of $\mathbb S^2$:
 \[\{\mathbf x\in\mathbb S^2\mid \mathbf p\cdot (\mathbf q\times \mathbf x)>0\},\quad \{\mathbf x\in\mathbb S^2\mid \mathbf p\cdot (\mathbf q\times \mathbf x)<0\}.\]
  On the other hand, recalling that $\zeta$ is rotationally invariant around $\mathbf q,$ we deduce that $\zeta$ is constant on the whole sphere. Taking into account the  zero mean condition
 \[\int_{\mathbb S^2}\zeta \dd{\sigma}=0,\]
  we get $\zeta\equiv 0.$ This is a contradiction to our assumption \eqref{qdp1}.

 \end{proof}

\begin{proof}[Proof of Theorem \ref{rgr}(ii)]

 As in the proof of (i),
 we can assume, without loss of generality, that $\mathbf m(\zeta)\neq\mathbf 0$ (since otherwise $\zeta\equiv 0$). Denote
 \begin{equation}\label{qd0i}
 \mathbf q:=\frac{\mathbf m(\zeta)}{|\mathbf m(\zeta)|}.
 \end{equation}
As in \eqref{skl1},   $\zeta\circ\mathsf R^{\mathbf q}_\theta$  is a maximizer of
   $E$ relative to $\mathcal R_\zeta\cap\mathcal C_\zeta$ for any $\theta\in\mathbb R$. Hence by Proposition \ref{ps42},
   \begin{equation}\label{abrt}
   \zeta\circ\mathsf R^{\mathbf q}_\theta-\zeta\in\mathbb E_2\quad\forall\,\theta\in\mathbb R.
   \end{equation}

  To prove the decomposition \eqref{decp}, we introduce the spherical coordinates $(u,v)$ on $\mathbb S^2$:
  \[\begin{cases}
  x_1=\cos v\cos u,\\
  x_2=\cos v\sin u,\\
  x_3=\sin v,
  \end{cases}\quad -\pi<u<\pi,\quad -\frac{\pi}{2}<v<\frac{\pi}{2}.
  \]
For simplicity, we only consider the special case $\mathbf q=\mathbf e_3$ (the general case can be proved in an analogous manner).  In this case, \eqref{abrt} becomes
\[\zeta(u+\theta, v)-\zeta(u, v)\in\mathbb E_2\quad\forall\,\theta\in\mathbb R,\]
from which we deduce that $\partial_u\zeta\in\mathbb E_2.$ On the other hand, in spherical coordinates,
\[\mathbb E_2= \mathrm{span}\{\cos^2v\sin(2u),\  \sin(2v)\cos u,\  \sin(2v)\sin u, \ \cos^2v\cos(2u), \ 3\sin^2v-1\}.\]
So
\[\partial_u\zeta=a\cos^2v\sin(2u)+b\sin(2v)\cos u+c\sin(2v)\sin u+d \cos^2v\cos(2u)+e(3\sin^2v-1)\]
for some $a,b,c,d,e\in\mathbb R,$
which further implies that
\[\zeta=-\frac{1}{2}a\cos^2v\cos(2u)+b\sin(2v)\sin u-c\sin(2v)\cos u+\frac{1}{2}d \cos^2v\sin(2u)+e(3\sin^2v-1)u+\zeta_z.\]
where $\zeta_z$ depends only on $v$. To ensure periodicity, it is necessary that $e=0.$ The desired decomposition then follows by taking
\[\zeta_e=-\frac{1}{2}a\cos^2v\cos(2u)+b\sin(2v)\sin u-c\sin(2v)\cos u+\frac{1}{2}d \cos^2v\sin(2u)\in\mathbb E_2.\]

   To finish the proof, we only need to show that $\mathbf q=\pm\mathbf p$ if $\beta\neq 0.$
   Repeating the argument as in
  \eqref{sfyse1},  we deduce that $\zeta\circ \mathsf R^{\mathbf q}_\theta$  satisfies
  \[\zeta\circ \mathsf R^{\mathbf q}_\theta=g(\mathcal G\zeta\circ \mathsf R^{\mathbf q}_\theta+\beta\mathbf p\cdot\mathbf x)\quad\forall\,\theta\in\mathbb R,\]
  which  yields
  \[\zeta =g(\mathcal G\zeta +\beta\mathbf p\cdot(\mathsf R^{\mathbf q}_\theta\mathbf x))\quad\forall\,\theta\in\mathbb R.\]
Taking into account   \eqref{nbv2}, we obtain
   \[g(\mathcal G\zeta+\beta\mathbf p\cdot\mathbf x)=g(\mathcal G\zeta+\beta\mathbf p\cdot(\mathsf R^{\mathbf q}_\theta\mathbf x))\quad\forall\,\theta\in\mathbb R.\]
   The remaining proof is almost identical to that  of  (i).

\end{proof}

\begin{proof}[Proof of Theorem \ref{rgr}(iii)]
Just repeat the argument  in the proof of  (i), with the  only modifications being the replacement of \eqref{rzta1} and  \eqref{fft1} by
 \[ \zeta=\zeta\circ\mathsf R^{\mathbf q}_\theta  \quad\forall\, |\theta|<\theta_0,\]
 \[ \max_{|\theta|<\theta_0}\left[\beta\mathbf p\cdot(\mathsf R^{\mathbf q}_\theta\mathbf x))-\beta\mathbf p\cdot\mathbf x\right]>0,\quad \min_{|\theta|<\theta_0}\left[\beta\mathbf p\cdot(\mathsf R^{\mathbf q}_\theta\mathbf x))-\beta\mathbf p\cdot\mathbf x\right]<0,\]
 respectively, where $\theta_0$ is a small positive number.
\end{proof}

\section{Rotating sphere}\label{sc6}

As we have mentioned in Section \ref{sec1}, if $\zeta$ solves \eqref{seee}, then $\zeta(\mathbf R^{\mathbf p}_{\beta t}\mathbf R^{\mathbf e_3}_{\Omega t}\mathbf x)$ is a solution of $(V_\Omega)$. In this section, we discuss the stability of such $\zeta$ under the dynamics of $(V_\Omega)$  based on the relation \eqref{v0v}.

For any function $v:\mathbb S^2\mapsto \mathbb R$, denote by ${\mathbb H}_v$ the set of all rigid rotations of $v$ about the polar axis, i.e.,
\[{\mathbb H}_v:=\left\{ v\circ\mathsf R\mid \mathsf R\in\mathbb H\right\},\]
where $\mathbb H$ is defined by \eqref{dfh1}.  Let $\mathcal M$ be a set of functions on $\mathbb S^2,$
 denote
 \[\mathbb H_{\mathcal M}:=\cup_{u\in\mathcal M}\mathbb H_u.\]

\begin{proposition}\label{relas}
Let $1<p<+\infty.$ Suppose that $\mathcal M\subset\mathring L^p(\mathbb S^2)$ is an $L^p$-stable set   under the dynamics of  $(V_0)$,
i.e., for any $\varepsilon>0,$ there exists some $\delta>0$, such that for any smooth solution $\zeta(t,\mathbf x)$ of $(V_0)$, it holds that
     \begin{equation}\label{scm501}
     \min_{v\in\mathcal M }\left\|\zeta(0,\cdot)-v\right\|_{L^p(\mathbb S^2)}<\delta\quad\Longrightarrow\quad  \min_{v\in\mathcal M }\left\|\zeta(t,\cdot)-v\right\|_{L^p(\mathbb S^2)}<\varepsilon\quad\forall\,t\in\mathbb R.
     \end{equation}
Then   $\mathbb H_{\mathcal M}$   is an $L^p$-stable set  under the dynamics of  $(V_\Omega)$.
\end{proposition}

\begin{proof}
  It suffices to show that for any sequence of solutions $\{\zeta^n\}$ of $(V_\Omega)$ and any sequence of times $\{t_n\}\subset\mathbb R$, if $\zeta^n(0,\cdot)$ converges to some $\xi\in \mathcal M$ in $L^p(\mathbb S^2),$ then, up to a subsequence, $\zeta^n(t_n,\cdot)$ converges to some $\eta\in \mathcal M$ in $L^p(\mathbb S^2).$
  Recalling \eqref{v0v}, we know that $\zeta^n(t, \mathbf R^{\mathbf e_3}_{-\Omega t}\mathbf x)$ is a solution of $(V_0)$. Since $\mathcal M$ is $L^p$-stable under the dynamics of $(V_0),$    it holds that, up to a subsequence,
      \[\left\|\zeta^n (t_n, \mathbf R^{\mathbf e_3}_{-\Omega t_n}\cdot)- \tilde \zeta \right\|_{L^p(\mathbb S^2)}\to 0 \]
      for some $\tilde\zeta\in\mathcal M,$
      which yields
      \[\left\|\zeta^n(t_n,\cdot)- \tilde \zeta ( \mathbf R^{\mathbf e_3}_{\Omega t_n}\cdot)\right\|_{L^p(\mathbb S^2)}\to0.\]
The desired result follows from the observation that  $\tilde \zeta ( \mathbf R^{\mathbf e_3}_{\Omega t_n}\mathbf x)$ converges to some $\eta\in \mathbb H_{\tilde\zeta}\subset\mathbb H_{\mathcal M}$.

\end{proof}

With Proposition \ref{relas}, we can discuss the stability of steady or rotating Euler flows related to \eqref{seee} under the dynamics of $(V_\Omega)$ for general $\Omega\in\mathbb R.$

\begin{theorem}
Let $1<p<\infty.$ Suppose that $\zeta\in \mathring L^p(\mathbb S^2)$ satisfies
 \begin{equation*}
\zeta= g(\mathcal G\zeta+\beta{\mathbf p}\cdot\mathbf x)\quad{\rm a.e.\, \,on } \,\,\mathbb S^2,
 \end{equation*}
for some function $g:\mathbb R\to\mathbb R\cup\{\pm \infty\}$, some unit vector  $\mathbf p\in\mathbb R^3$ and some $\beta\in\mathbb R.$
\begin{itemize}
  \item [(i)]If $g$ is decreasing, then $\mathbb H_\zeta$  is  an  $L^p$-stable  set under the dynamics of $(V_\Omega)$.  Moreover, if additionally $\beta\neq 0$ and $\mathbf p\parallel\mathbf e_3,$ then $\mathbb H_\zeta=\{\zeta\}$, and thus $\zeta$  is    $L^p$-stable   under the dynamics of $(V_\Omega)$.
  \item  [(ii)]   If $g\in C^1(\mathbb R)$ satisfies
 \[ 0\leq  g'\leq 6,\]
 then
  $\mathbb H_{(\zeta+\mathbb E_2)\cap \mathcal R_{\zeta}}$   is an    $L^p$-stable  set under the dynamics of $(V_\Omega)$.
     Moreover, if  additionally $\beta\neq 0$ and $\mathbf p\parallel\mathbf e_3,$ then
     \[\mathbb H_{(\zeta+\mathbb E_2)\cap \mathcal R_{\zeta}}= (\zeta+\mathbb E_2)\cap \mathcal R_{\zeta},\] and thus $ (\zeta+\mathbb E_2)\cap \mathcal R_{\zeta}$  is    an    $L^p$-stable  set under the dynamics of $(V_\Omega)$.
  \item  [(iii)] If $g\in C^1(\mathbb R)$ is increasing and $-\Delta-g'(\mathcal G \zeta+\beta{\mathbf p}\cdot\mathbf x)>0$ on $\mathbb E_1^\perp,$ then $\mathbb H_\zeta$  is  an  $L^p$-stable  set under the dynamics of $(V_\Omega)$.
 Moreover, if   additionally  $\beta\neq 0$ and $\mathbf p\parallel\mathbf e_3,$ then $\mathbb H_\zeta=\{\zeta\}$, and thus $\zeta$  is    $L^p$-stable   under the dynamics of $(V_\Omega)$.
\end{itemize}

\end{theorem}

\bigskip

 \noindent{\bf Acknowledgements:}  We are very grateful to the two anonymous referees for their valuable comments and suggestions.
 D. Cao was supported by National Key R\&D Program (Grant 2023YFA1010001) and NNSF of China (Grant
12371212).  G. Wang was supported by NNSF of China (Grant 12471101) and  Fundamental Research Funds for the Central Universities (Grant DUT23RC(3)077).

\bigskip
\noindent{\bf  Data Availability} Data sharing not applicable to this article as no datasets were generated or analyzed during the current study.

\bigskip
\noindent{\bf Declarations}

\bigskip
\noindent{\bf Conflict of interest}  The authors declare that they have no conflict of interest to this work.

\phantom{s}
 \thispagestyle{empty}

\end{document}